\documentclass{amsart}
\usepackage{amssymb,amscd,verbatim,latexsym}
\usepackage{enumerate}
\usepackage{multicol}
\usepackage[mathscr]{eucal}

\usepackage{color}
\newcommand\ede{ \, := \, }

\newcommand\lp{`}
\newcommand\rp{'}

\newcommand\dlp{``}
\newcommand\drp{''}

\newcommand\datver[1]{\def\datverp
{\par\boxed{\boxed{\text{Version: #1; Run: \today}}}}}\datver{0.1}

\newcommand{\mfk}{\mathfrak}
\newcommand{\mfkg}{\mathfrak g}

\newcommand\pullback{\sp{\downarrow\downarrow}}
\newcommand{\tto}{\rightrightarrows}

\newcommand{\RR}{\mathbb R}

\newcommand{\ZZ}{\mathbb Z}

\newcommand{\CI}{{\mathcal C}^{\infty}}
\newcommand{\CIc}{{\mathcal C}^{\infty}_{\text{c}}}
\newcommand\pa{{\partial}}

\newcommand{\Aut}{\operatorname{Aut}}

\newcommand\ssub{stratified tame submersion}
\newcommand{\edge}{\maE}

\newcommand{\maE}{\mathcal E}

\newcommand{\maG}{\mathcal G}
\newcommand{\maH}{\mathcal H}

\newcommand{\maK}{\mathcal K}

\newcommand{\maP}{\mathcal P}

\newcommand{\maT}{\mathcal T}
\newcommand{\maV}{\mathcal V}
\newcommand{\maW}{\mathcal W}

\newcommand\m[1]{${#1}$}

\newtheorem{theorem}{Theorem}[section]
\newtheorem{proposition}[theorem]{Proposition}
\newtheorem{corollary}[theorem]{Corollary}

\newtheorem{lemma}[theorem]{Lemma}
\newtheorem{notation}[theorem]{Notations}
\theoremstyle{definition}
\newtheorem{definition}[theorem]{Definition}
\theoremstyle{remark}
\newtheorem{remark}[theorem]{Remark}
\newtheorem{example}[theorem]{Example}

\author[V. Nistor]{Victor Nistor} \address{Universit\'{e} de Lorraine,
  UFR MIM, Ile du Saulcy, CS 50128, 57045 METZ, France and
  Inst. Math. Romanian Acad.  PO BOX 1-764, 014700 Bucharest Romania} 
\email{victor.nistor@univ-lorraine.fr} 

\thanks{
  V. Nistor has been partially supported by
  ANR-14-CE25-0012-01 (SINGSTAR).\\
Manuscripts available from {\bf http:{\scriptsize
    //}iecl.univ-lorraine.fr{\scriptsize
    /}$\tilde{}$Victor.Nistor{\scriptsize /}}}

\date\today

\begin{document}

\title[Groupoids]{Desingularization of Lie groupoids and 
pseudodifferential operators on singular spaces}

\begin{abstract}  
We introduce and study a ``desingularization'' of a Lie groupoid $\maG$ along an
``$A(\maG)$-tame'' submanifold $L$ of the space of units $M$.  An
$A(\maG)$-tame submanifold $L \subset M$ is one that has, by
definition, a tubular neighborhood on which $A(\maG)$ becomes a thick
pull-back Lie algebroid. The construction of the desingularization
$[[\maG:L]]$ of $\maG$ along $L$ is based on a canonical fibered
pull-back groupoid structure result for $\maG$ in a neighborhood of
the tame $A(\maG)$-submanifold $L \subset M$. This local structure result is
obtained by integrating a certain groupoid morphism, using results of
Moerdijk and Mrcun (Amer. J. Math. 2002). Locally, the desingularization
$[[\maG:L]]$ is defined using a construction of Debord and Skandalis
(Advances in Math., 2014). The space of units of
the desingularization $[[\maG:L]]$ is $[M:L]$, the blow up of $M$
along $L$. The space of units and the desingularization groupoid 
$[[\maG:L]]$ are constructed using a gluing construction of
Gualtieri and Li (IMRN 2014). 
We provide an explicit description of the structure of the
desingularized groupoid and we identify its Lie algebroid, which is
important in analysis applications. We also discuss a variant of our 
construction that is useful for analysis on asymptotically hyperbolic
manifolds. We conclude with
an example relating our constructions to the so called ``edge
pseudodifferential calculus.'' The paper is written such that it
also provides an introduction to Lie groupoids designed for
applications to analysis on singular spaces.
\end{abstract}

\maketitle

\tableofcontents

\section*{Introduction}

Motivated by certain questions in analysis on singular spaces, we
introduce and study the desingularization of a Lie groupoid $\maG$
with respect to a tame submanifold $L$ of its set of units $M$. More
precisely, let $A\to M$ be a Lie algebroid over a manifold with
corners $M$ and let $L \subset M$ be a submanifold. Recall that $L$ is
called {\em $A$-tame} if it has a tubular neighborhood $\pi : U \to L$
in $M$ such that the restriction $A\vert_{U}$ is isomorphic to the
thick pull-back Lie algebroid $\pi\pullback (B)$, for some Lie
algebroid $B \to L$. Let $\maG$ be a Lie groupoid with units $M$ and
Lie algebroid $A(\maG)$.  Let $L \subset M$ be an $A(\maG)$-tame
submanifold.  In this paper, we define and study a new Lie groupoid
$[[\maG: L]]$, called the \dlp desingularization\drp\ of $\maG$ along
$L$. The desingularization $[[\maG:L]]$ has units $[M:L]$, the blow-up
of $M$ along $L$, and plays in the category of Lie algebroids a role
similar to the role played by the usual (real) blow-up in the category
of manifolds with corners.

Let us try now to give a quick idea of this desingularization
procedure, the full details being given in the main body of the paper.
Let $\maG$ is a Lie groupoid with units $M$ (we write $\maG \tto M$)
and let $L \subset M$ be an $A(\maG)$-tame submanifold with tubular
neighborhood $\pi : U \to L$. In particular, the blow-up $[M:L]$ is
also defined if $L$ is tame. Let us also assume the fibration $\pi : U
\to L$ to be a ball bundle over $L$. The reduction groupoid
$\maG_{U}\sp{U}$ will then have a fibered pull-back groupoid structure
on $U$ (Theorem~\ref{thm.tame}), and hence we can replace it with a
slight modification of the adiabatic groupoid
to define the {\em desingularization} $[[\maG:L]]$ of 
$\maG$ along $L$. To this end, we use also a gluing 
construction due to Gualtieri and Li \cite{Gualtieri}.

Our definition of the desingularization of a Lie groupoid with respect
to a tame submanifold is motivated by the method of successively
blowing-up the lowest dimensional strata of a singular space, which
was successfully used in the analysis on singular spaces.  The
successive blow up of the lowest dimensional singular strata of a
(suitable) singular space leads to the eventual removal of all
singularities. This approach was used in \cite{BMNZ} to obtain a
well-posedness result for the Poisson problem in weighted Sobolev
spaces on $n$-dimensional polyhedral domains using energy methods (the
Lax-Milgram lemma). One would like to use also other methods than the
energy method to study singular spaces, such as the method of layer
potentials, but then one has to study the resulting integral kernel
operators.

In fact, our definition of desingularization groupoid provides the
necessary results for the construction of integral kernel operators on
the resulting blown-up spaces, since the kernels of the resulting
integral operators will be defined on the groupoid. It turns out that
quite general operators can be obtained using invariant
pseudodifferential operators on the groupoid \cite{aln1, ASkandalis2,
  Monthubert, NWX}. For instance, by combining this desingularization
construction with the construction of psedodifferential operators on
groupoids, one can recover the pseudodifferential calculi of Grushin
\cite{grushin71}, Mazzeo \cite{mazzeo91}, and Schulze
\cite{schulzeEdge, SchulzeBook91}.

A groupoid $\maG \tto M$ (that is, a groupoid with units $M$), can be
used to model the analysis on $M$, which is our main interest. While
our desingularization procedure is the groupoid counterpart of the
blow-up of $M$ with respect to a tame submanifold, it is the later
that is our main interest.  This leads necessarily to manifolds with
corners, as follows: the blow-up of a smooth manifold with respect to
a submanifold is a manifold with boundary, but the blow-up of a
manifold with boundary along a tame submanifold is a manifold with
corners. In general, the blow-up of a manifold with corners with
respect to a tame submanifold is a manifold with corners of higher
maximum codimension (i.e. rank).  Thus, even if one is interested in
analysis on smooth manifolds, sometimes one is lead to consider also
manifolds with corners. See, for example, \cite{BMNZ, daugeBook, kottkeMelrose,
  NP} for some motivation and futher references. This paper will thus 
provide the background for the
construction of the integral kernel (or pseudodifferential) operators
on the resulting blown-up spaces.

The paper is organized as follows. The first section is devoted mostly
to background material. We thus review manifolds with corners and tame
submersions and establish a canonical (i.e. fibration) local form for
a tame submersion that generalizes to manifolds with corners the
classical result in the smooth case. We then recall the definitions of
a Lie groupoid, of a Lie algebroid, and of the Lie algebroid
associated to a Lie groupoid. We do that in the framework that we
need, that is, that of manifolds with corners. Almost everything
extends to the setting of manifolds with corners without any
significant change. One must be careful, however, to use {\em tame
  fibrations}. One of the main results of this paper is the
construction of the desingularization of a Lie groupoid $\maG$ along
an $A(\maG)$-tame submanifold. This requires several other,
intermediate constructions, such as that of the adiabatic
(deformation) groupoid and of the thick pull-back Lie algebroid. In
the second section, we thus review and extend all these examples as
well as other, more basic ones that are needed in the construction of
the desingularization groupoid. In particular, we introduce the so
called \dlp edge modification\drp\ of a groupoid using results of
Debord and Skandalis \cite{DebordSkandalis}.  We shall need a gluing
construction due to Gualtieri and Li \cite{Gualtieri}, which we also
review and extend to our setting. The third section contains most of our
main results. We first prove a local structure theorem for a Lie
groupoid $\maG$ with units $M$ in a tubular neighborhood $\pi : U \to
L$ of an $A(\maG)$-tame submanifold $L \subset M$ using results on the
integration of Lie algebroid morphisms due to Moerdijk and Mrcun
\cite{MoerdijkMrcun02}.  More precisely, we prove that the reduction
of $\maG$ to $U$ is isomorphic to $\pi\pullback (\maG_L\sp{L})$, the
fibered pull-back groupoid to $U$ of the reduction of $\maG$ to
$L$. This allows us to define the desingularization first for this
type of fibered pull-back groupoids.  We identify the Lie algebroid of
the desingularization as the desingularization of its Lie algebroid
(the desingularization of a Lie algebroid was introduced in
\cite{schroedinger}). We also introduce an {\em anisotropic} version of
the desingularized groupoid and determine its Lie algebroid as well. 
We conclude with an example related to the \lp edge\rp-calculus
(see \cite{krainer2014} and the references therein).

The paper is written such that it provides also an introduction to Lie
groupoids for students and researchers interested in their
applications to analysis on singular spaces. This is the reason for
which the first two sections contain additional material that will
help people interested in analysis understand the role of
groupoids. For instance, we discuss the convolution algebras of some
classes of Lie groupoids. We also provide most of the needed
definitions to make the paper as self-contained as possible. We also
study in detail the many needed examples.

\subsection*{A note on notation and terminology} We shall use manifolds
with corners extensively.  They are defined in Subsection
\ref{ssec.mc}.  A manifold without corners will be called {\em
  smooth}.  We take the point of view that all maps, submanifolds, and
so on will be defined in the same way in the corner case as in the
smooth case, except that all our submanifolds will be assumed to be
{\em closed}. Sometimes, we need maps and submanifolds with special
properties, they will usually be termed \dlp tame\drp, for instance, a
tame submersion of manifolds with corners will have the property that
all its fibers are smooth manifolds. This property is not shared by
general submersions, however. Also, we use only {\em real} vector
bundles and functions, to avoid confusion and simplify notation. The
results extend without any difficulty to the complex case, when one
wishes so.

Moreover, all our manifolds will be paracompact, but we do not require
them to be Hausdorff in general. However, all the spaces of units of
groupoids and the bases of Lie groupoids will be Hausdorff.

\subsection*{Acknowledgements} 
We thank Daniel and Ingrid Belti\c{t}\u{a} and 
Kirill Mackenzie for useful discussions.

\section{Preliminaries on Lie algebroids}
\label{sec.LgLa}
We now recall the needed definitions and properties of Lie groupoids
and of Lie algebroids. We shall work with manifolds with corners, so
we also recall some basic definitions and results on manifolds with
corners.  Few results in this section are new, although the
presentation probably is. We refer to Mackenzie\rp s books
\cite{MackenzieBook1, MackenzieBook2} for a nice introduction to the subject, as well
as to further references and historical comments on of Lie groupoids
and Lie algebroids. See also \cite{buneciSurvey, marcutzActa, MoerdijkFolBook} for
the more specialized issues relating to the applications envisioned
in this paper.

\subsection{Manifolds with corners and notation}\label{ssec.mc}
In the following, by a {\em manifold}, we shall mean a $\CI$-manifold,
possibly with corners. By a {\em smooth manifold} we shall mean a
$\CI$-manifold {\em without corners.} All our manifolds will be
assumed to be paracompact.  Recall \cite{Joyce, MargalefHandbook,
  kottkeMelrose} (and the references therein) 
  that $M$ is a {\em manifold with corners} of
dimension $n$ if it is locally diffeomorphic to an open subset of
$[-1, 1]\sp{n}$ with smooth changes of coordinates.  A point $p \in M$
is called of {\em depth} $k$ if it has a neighborhood $V_p$
diffeomorphic to $[0, a)^{k} \times (-a, a)^{n-k}$, $a > 0$, by a
  diffeomorphism $\phi_p : V_p \to [0, a)^{k} \times (-a, a)^{n-k}$
    mapping $p$ to the origin:\ $\phi_p(p) = 0$. Such a neighborhood
    will be called {\em standard}. A function $f : M \to M_1$ between
    two manifolds with corners will be called {\em smooth} if it its
    components are smooth in all coordinate charts.

A connected component $F$ of the set of points of depth $k$ will be
called an {\em open face (of codimension $k$)} of $M$.  The maximum
depths of a point in $M$ will be called the {\em rank} of $M$. Thus
the smooth manifolds are exactly the manifolds of rank zero. The
closure in $M$ of an open face $F$ of $M$ will be called a {\em
  closed} face of $M$. The closed faces of $M$ may not be manifolds
with corners on their own.

We define the tangent space to a manifold with corners $TM$ as usual,
that is, as follows: the vector space $T_pM$ is the set of derivations
$D_p : \CI(M) \to \RR$ satisfying $D_p(fg) = f(p) D_p(g) + D_p(f)
g(p)$ and $TM$ is the disjoint union of the vector spaces $T_pM$, with
$p \in M$. Let $v$ be a tangent vector to $M$ (say $v \in T_pM$). We
say that $v$ is {\em inward pointing} if, by definition, there exists
a smooth curve $\gamma : [0, 1] \to M$ such that $\gamma\rp(0) = v$
(so $\gamma(0) = p$). The set of inward pointing vectors in $v \in
T_{x}(M)$ will form a closed cone denoted $T_{x}\sp{+}(M)$. If, close
to $x$, our manifold with corners is given by the conditions $\{
f_i(y) \ge 0\}$ with $df_i$ linearly independent at $x$, then the cone
$T_{x}\sp{+}(M)$ is given by
\begin{equation}
  T_{x}\sp{+}(M) \, = \, \{ v \in T_{x}M,\, df_i(v) \ge 0 \}.
\end{equation}

Let $M$ and $M_1$ be manifolds with corners and $f : M_1 \to M$ be a
smooth map. Then $f$ induces a vector bundle map $df : TM_1 \to TM$,
as in the smooth case, satisfying also $df(T_z\sp{+}(M_1)) \subset
T_{f(z)}\sp{+}M$. If the smooth map $f : M_1 \to M$ is injective, has
injective differential $df$, and has closed range, then we say that
$f(M_1)$ is a (closed) {\em submanifold} of $M$.
All our submanifolds will be closed, so we shall simply say \dlp
manifold\drp\ instead of \dlp closed manifold.\drp\ We are thus
imposing the least restrictions on smooth maps and submanifolds,
unlike \cite{Joyce}, for example. For example, a smooth map $f$
between manifolds with corners is a {\em submersion} if, by
definition, the differential $df = f_*$ is surjective (as in the case
of smooth manifolds). However, we will typically need a special class
of submersions with additional, properties, the {\em tame
  submersions}. More precisely, we have the following definition.

\begin{definition}\label{def.tame submersion} 
A {\em tame submersion} $h$ between two manifolds with corners $M_1$
and $M$ is a smooth map $h : M_1 \to M$ such that its differential
$dh$ is surjective everywhere and
\begin{equation*}
 (dh_x)\sp{-1} (T_{h(x)}\sp{+}M) \, = \, T_{x}\sp{+}M_1 \,.
\end{equation*}
(That is, $dh(v)$ is an inward pointing vector of $M$ if, and only if,
$v$ is an inward pointing vector of $M_1$.)
\end{definition}

We do not require our tame submersions to be surjective (although, as
we will see soon below, they are open, as in the smooth case). We
shall need the following lemma.

\begin{lemma}\label{lemma.depth}
 Let $h : M_1 \to M$ be a tame submersion of manifolds with corners.
 Then $x$ and $h(x)$ have the same depth.
\end{lemma}

\begin{proof}
 This is because the depth of $x$ in $M$ is the same as the depth of
 the origin $0$ in $T_{x}\sp{+}M_1$, which, in turn, is the same as
 the depth of the origin $0$ in $T_{h(x)}\sp{+}M$ since $dh_x$ is
 surjective and $(dh_x)\sp{-1} (T_{h(x)}\sp{+}M) = T_{x}\sp{+}M_1$.
\end{proof}

The following lemma is probably known, but we could not find a
suitable reference.

\begin{lemma}\label{lemma.corners}
Let $h : M_1 \to M$ be a tame submersion of manifolds with corners. 
\begin{enumerate}[(i)]
 \item The rank of $M_1$ is $\leq $ the rank of $M$.
 \item For $m_1 \in M_1$, there exists an open neighborhood $U_1$ of
   $m_1$ in $M_1$ such that $U:=h(U_1)$ is open and the restriction of
   $h$ to $U_1$ is a fibration $U_1 \to U$.
 \item Let $L \subset M$ be a submanifold, then $L_1 := h\sp{-1}(L)$ is
a submanifold of $M_1$ of rank $\le$ the rank of $L$.
\end{enumerate} 
\end{lemma}

\begin{proof}
We have already noticed that the depths of $x$ and $h(x)$ are the same
(Lemma \ref{lemma.depth}), so the rank of $M_1$, which is the maximum
of the depths of $x \in M_1$, is inferior or equal to the rank of
$M$. This proves (i).

Let us now prove (ii). Let $m_1 \in M_1$ be of depth $k$.  We can
choose a standard neighbourhood $W_1$ of $m_1$ in $M_1$ and a standard
neighborhood $W$ of $h(m_1)$ in $M$ such that $h(W_1) \subset
W$. Since our problem is local, we may assume that $M_1 = W_1 = [0,
  a)^k \times (-a, a)^{n_1-k}$ and $M = W = [0, b)^k \times (-b, b)^{n
      -k}$, $a,b > 0$, with $m_1$ and $h(m_1)$ corresponding to the
    origins. Note that both $M$ and $M_1$ will then be manifolds with
    corners of rank $k$; this is possible since $h$ preserves the
    depth, by Lemma \ref{lemma.depth}. We can then extend $h$ to a map
    $h_0 : Y_1 := (-a, a)^{n_1} \to \RR^{n}$ that is a (usual)
    submersion at $0 = m_1$ (not necessarily tame). By decreasing $a$,
    if necessary, we may assume that $h_0$ is a (usual) submersion
    everywhere and hence that $h_0(Y_1)$ is open in $\RR^{n}$. By
    standard differential geometry results, we can then choose an open
    neighborhood $V$ of $0 = h_0(m_1)$ in $\RR\sp{n}$ and an open
    neighborhood $V_1$ of $0 = m_1$ in $Y_1 := (-a, a)^{n_1}$ such
    that the restriction $h_1$ of $h_0$ to $V_1$ is a fibration $h_1:
    V_1 \to V$ with fibers diffeomorphic to $(-1, 1)\sp{n_1-n}$.  By
    further decreasing $V$ and $V_1$, we may assume that $V$ is an
    open ball centered at 0.

Next, we notice that $M \cap V$ consists of the vectors in $V$ that
have the first $k$ components $\ge 0$. By construction, we therefore
have that
\begin{equation*}
 h_1(M_1 \cap V_1) \, = \, h_0(M_1 \cap V_1) \ \subset \ M \cap V \, = \,
 \Big (\, [0, b)^k \times (-b, b)^{n -k} \, \Big ) \cap V \,.
\end{equation*}
Let $U_1 := M_1 \cap V_1$. We will show that we have in fact more,
namely, that we have
\begin{equation}\label{eq.one}
 U_1 \, = \, h_1\sp{-1}(M \cap V)\ \ \mbox{ and } \ \ 
 h_1(U_1) = M \cap V\,, 
\end{equation}
which will prove (ii) for $U_1 := M_1 \cap V_1$, since $h_1 : V_1 \to
V$ is a fibration with fibers diffeomorphic to $(-1, 1)\sp{n_1-n}$ and
$h(U_1) = h_1(U_1) = M \cap V$ is open in $M$.

Indeed, in order to prove the relations in Equation \eqref{eq.one},
let us notice that, since $h_1$ is surjective, it is enough to prove
that $U_1 = h_1\sp{-1}(M \cap V)$, since that will then give, right
away that $h_1(U_1) = M \cap V$. The relations in Equation
\eqref{eq.one} will be enough to complete the proof of (ii). Let us
assume then, by contradiction, that it is not true that $U_1 =
h_1\sp{-1}(M \cap V)$. This means that there exists $p = (p_i) \in V_1
\smallsetminus M_1$ such that $h_1(p) = h_0(p) \in M \cap V = \big(
[0, b)^k \times (-b, b)^{n -k} \big) \cap V$.  Let us choose $q =
  (q_i)$ in $M_1 \cap V_1$ of depth zero. That is, we assume that $q$
  is an interior point of $M_1 \cap V_1$.  Then the two points $h_1(p)
  = h_0(p)$ and $h_1(q) = h_0(q) = h(q)$ both belong to $M$, more
  precisely,
\begin{equation*}
 h_1(p), h_1(q) \, \in \, M \cap V \, = \, \big( [0, b)^k \times (-b,
   b)^{n -k} \big) \cap V \,,
\end{equation*}
which is the first octant in a ball.  Therefore $h_1(p)$ and $h_1(q)$
can be joined by a path $\gamma = (\gamma_i) : [0, 1] \to M \cap V$,
with $\gamma(1) = h_1(p)$ (and hence with $\gamma(0) = h_1(q)$).  All
paths are assumed to be continuous, by definition. Since $h$ preserves
the depth, $h_1(q) = h_0(q) = h(q)$ is moreover an interior point of
$M \cap V$.  Therefore we may assume that the path $\gamma(t)$
consists completely of interior points of $M$ for $t < 1$.

We can lift the path $\gamma$ to a path $\tilde \gamma : [0, 1] \to
V_1$ with $\tilde \gamma(0) = q$, $\tilde \gamma(1) = p$, $\gamma =
h_1 \circ \tilde \gamma$, since
\begin{equation*}
 h_1 := h_0\vert_{V_1} : V_1 \to V
\end{equation*}
is a fibration. We have $\tilde \gamma_i(0) = q_i> 0$ for $i = 1,
\ldots, k$, since $q = (q_i)$ is an interior point of $V_1 \cap M_1$.
On the other hand, since $p \notin M_1$, there exists at least one
$i$, $1 \le i \le k$, such that $\tilde \gamma_i(1) = p_i < 0$. Since
$\tilde \gamma_i(0) = q_i > 0$ and the functions $\tilde \gamma_j$ are
continuous, we obtain that the set
\begin{equation*}
 Z \ede \cup_{j=1}\sp{n} \tilde \gamma_j\sp{-1}(0) \, = \, \{\, t \in
 [0, 1], \ \mbox{ there exists } \ 1 \le j \le k \ \mbox{ such that }
 \ \tilde \gamma_j(t) = 0 \, \}
\end{equation*}
is closed and non-empty. Let $t_* = \inf Z \in Z$. Then $t_* > 0$
since $q = (q_i) = (\tilde \gamma_i(0))$ is of depth zero, meaning
that $\tilde \gamma_j(0)>0$ for $1 \le j \le k$, and hence that $0
\notin Z$.  Using again $\tilde \gamma_j(0)>0$, we obtain $\tilde
\gamma_i(s) > 0$ for all $0 \le s < t_*$, by the minimality of $t_*$,
since the functions $\tilde \gamma_j$ are continuous.  Hence $\tilde
\gamma(s) \in M_1 \subset Y_1$ for $s < t_*$.  (Recall that $h_0 : Y_1
:= (-a, a)^{n_1} \to \RR^{n}$ and that we are assuming $M_1 = [0,
  1)\sp{k} \times (-1, 1)\sp{n-k}$.)  We obtain that $\tilde
  \gamma(t_*) \in M_1 \cap V_1$, because $M_1$ is closed in $Y_1$.
  Therefore $t_* < 1$, because $p = \tilde \gamma(1) \notin M_1$.
  Since $\tilde \gamma_j(t_*) = 0$ for some $j$, we have that $\tilde
  \gamma(t_*)$ is a boundary point of $M_1$, and hence it has depth
  $>0$. Hence the depth of $\gamma(t_*) = h_0(\tilde \gamma(t_*)) =
  h(\tilde \gamma(t_*))$ is also $>0$ since $h$ preserves the
  depth. But this is a contradiction since $\gamma(t)$ was constructed
  to consist entirely of interior points for $t < 1$. This proves
  (ii).

The last part is a consequence of (ii), as follows. We use the same
notation as in the proof of (ii). We may assume $h\sp{-1}(L)$ to be
non-empty, because otherwise the statement is obviously true.  Let us
choose then $m_1 \in h\sp{-1}(L)$ and denote $m = h(m_1)$.  By the
statement (ii) just proved, there exit neighborhoods $U_1$ of $m_1$
and $U$ of $m$ such that the restriction of $h$ to $U_1$ induces a
fibration $h_2 := h\vert_{U_1} : U_1 \to U$. By decreasing $U_1$ and
$U$, if necessary, we can assume that the fibers of $h_2$ are
diffeomorphic to $(-1, 1)\sp{n - n\rp}$. Let $V_1$ be a standard
neighborhood of $m = h(m_1)$ in $L$. Then $h_2\sp{-1}(V_1)$ is a
standard neighborhood of $m_1$ in $h\sp{-1}(L)$. This completes the
proof of (iii) and, hence, also of the lemma.
\end{proof}

We shall use the above result in the following way:

\begin{corollary}\label{cor.corners}
Let $h : M_1 \to M$ be a tame submersion of manifolds with corners.
\begin{enumerate}[(i)]
\item $h$ is an open map.
\item The fibers $h\sp{-1}(m)$, $m \in M$, are smooth manifolds (that
  is, they have no corners).
\item Let us denote by $\Delta \in M \times M$ be the diagonal and by
  $h \times h : M_1 \times M_1 \to M \times M$ the product map $h
  \times h(m, m\rp) = (h(m), h(m\rp))$.  Then $(h \times
  h)\sp{-1}(\Delta)$ is a submanifold of $M_1 \times M_1$ of the same
  rank as $M_1$.
\end{enumerate} 
\end{corollary}

\begin{proof}
The first part follows from Lemma \ref{lemma.corners}(ii).  The second
and third parts follow from Lemma \ref{lemma.corners}(iii), by taking
$L = \{m\}$ for (ii) and $L = \Delta$ for (iii).
\end{proof}

We shall use the following conventions and notations.

\begin{notation}\label{not.Gamma}{\normalfont
If \m{E \to X} is a smooth vector bundle, we denote by \m{\Gamma(X;
  E)} (respectively, by \m{\Gamma_c(X;E)}) the space of smooth
(respectively, smooth, compactly supported) sections of
\m{E}. Sometimes, when no confusion can arise, we simply write
\m{\Gamma(E)}, or, respectively, \m{\Gamma_c(E)} instead of $\Gamma(X;
E)$, respectively $\Gamma_c(X; E)$. If $M$ is a manifold with corners,
we shall denote by
\begin{equation*}
 \maV_b(M) \, := \, \{ \, X \in \Gamma(M; TM), \, X \mbox{ tangent to
 all faces of } M \, \}
\end{equation*}
the set of vector fields on $M$ that are tangent to all faces of $M$
\cite{kottkeMelrose}.}
\end{notation}

For further reference, let us recall a classical result of Serre and
Swan \cite{karoubiBook}, which we formulate in the way that we will
use.

\begin{theorem}[Serre-Swan, \cite{karoubiBook}]\label{SS}
 Let $M$ be a compact Hausdorff manifold with corners and $\maV$ be a
 finitely generated, projective $\CI(M)$-module. Then there exists a
 real vector bundle $E_{\maV} \to M$, uniquely determined up to
 isomorphism, such that $\maV \simeq \Gamma(M; E_{\maV})$ as
 $\CI(M)$-module.  We can choose $E_{\maV}$ to depend functorially on
 $\maV$, in particular, any $\CI(M)$-module morphism $f : \maV \to
 \maW \simeq \Gamma(M; E_{\maW})$ induces a unique smooth vector
 bundle morphism $\tilde f : E_{\maV} \to E_{\maW}$ compatible with
 the isomorphisms $\maV \simeq \Gamma(M; E_{\maV})$ and $\maW \simeq
 \Gamma(M; E_{\maW})$.
\end{theorem}

In particular, there exists a (unique up to isomorphism) vector bundle
$T\sp{b}M$ such that $\Gamma(T\sp{b}M) \simeq \maV_b(M)$ as
$\CI(M)$-modules \cite{kottkeMelrose}, where $\maV_b$ is as introduced
in \ref{not.Gamma}.

\subsection{Definition of Lie groupoids and Lie algebroids}
\label{ssec.dLgLa} Recall that a {\em groupoid} $\maG$ is a small category 
in which every morphism is invertible. The class of objects of $\maG$,
denoted $\maG\sp{(0)}$, is thus a set. For convenience, we shall
denote $M := \maG\sp{(0)}$. The set of morphisms $\maG :=
\maG\sp{(1)}$ is thus also a set.

One typically thinks of a groupoid in terms of its structural
morphisms.  Thus the domain and range of a morphism give rise to maps
$d, r : \maG \to M$. We shall therefore write $d, r: \maG \tto M$ (or,
simply, $\maG \tto M$) for a groupoid with units $M$. We shall denote
by $\mu(g, h) = gh$ the composition of two composable morphisms $g$
and $h$, that is, the composition of two morphisms satisfying $d(g) =
r(h)$ and by
\begin{equation}
 \maG\sp{(2)} \, := \, \{(g, h) \in \maG \times \maG, \ d(g) = r(h) \}
\end{equation}
the domain of the composition map $\mu$. Let us notice that, by
Corollary \ref{cor.corners}, the set $\maG\sp{(2)}$ is a manifold with
corners whenever $M$ and $\maG$ are manifolds with corners and $d$ and
$r$ are tame submersions of manifolds with corners.

The objects of $\maG$ will also be called {\em units} and the
morphisms of $\maG$ will also be called {\em arrows}. To the groupoid
$\maG$ there are also associated the inverse map $i(g) = g\sp{-1}$ and
the embedding $u : M \to \maG$, which associates to each object its
identity morphism. If $M$ and $\maG$ are manifolds with corners and
$i$ is smooth and $d$ is a tame submersion of manifolds with corners,
then $r$ is also a tame submersion of manifolds with corners.

For simplicity, we typically write $gh := \mu(g,h)$.  The structural
morphism $d, r, \mu, i, u$ will satisfy the following conditions
\cite{buneciSurvey, MackenzieBook2, MoerdijkMrcun02}:
\begin{enumerate}
 \item $g_1(g_2 g_3) = (g_1 g_2) g_3$ for any $g_i \in \maG$ such that
   $d(g_i) = r(g_{i+1})$.
 \item $g u(d(g)) = g$ and $u(r(g)) g = g$ for any $g \in \maG$.
 \item $g i(g) = u(r(g))$ and $i(g) g = u(d(g))$ for any $g \in \maG$.
\end{enumerate}

Recall then

\begin{definition} A {\em Lie groupoid} is a groupoid $\maG \tto M$ such that
\begin{enumerate}
  \item $M$ and $\maG$ are manifolds (possibly with corners), with $M$
    Hausdorff,
  \item the structural morphisms $d, r, i, u$ are smooth,
  \item $d$ is a tame submersion of manifolds with corners (so
    $\maG\sp{(2)}$ is a manifold) and $\mu : \maG\sp{(2)} \to \maG$ is
    smooth.
\end{enumerate}
\end{definition}

Note that we do not assume $\maG = \maG\sp{(1)}$ to be Hausdorff,
although that will be the case for most groupoids considered in this
paper.  Lie groupoids were introduced by Ehresmann. See
\cite{MackenzieBook2} for a comprehensive introduction to the subject
as well as for more references.  Note that $\maG$ is {\em not}
required to be Hausdorff, as this will needlessly remove a large class
of important examples, such as the ones arising in the study of
foliations \cite{ConnesBook}.

We are interested in Lie groupoids since many operators of interest
have distribution kernels that are naturally defined on a Lie
groupoid. Let us see how this is achieved in the case of regularizing
operators. Let $\maG \tto M$ be a Lie groupoid and let us choose a
metric on $A(\maG)$.  We can use this metric and the projections $r :
\maG_x \to M$ that satisfy $T\maG_x \simeq r\sp{*}(A(\maG))$ to obtain
a family of metrics $g_x$ on $\maG_x$. By constructions, these metrics
will be right invariant.  Whenever integrating on a set of the form
$\maG_x$, $x \in M$, we shall do that with respect to the volume form
associated to $g_x$. Let us assume, for simplicity that $\maG$ is
Hausdorff. We then define a convolution product on $\CIc(\maG)$ by the
formula
\begin{equation}\label{eq.def.convolution}
 \phi \ast \psi (g) \ede \int_{\maG_{d(g)}} \phi(gh\sp{-1})
 \psi(h)dh\,.
\end{equation}

A {\em subgroupoid} of a groupoid $\maG$ is a subset $\maH$ such that
the structural morphisms of $\maG$ induce a groupoid structure on
$\maH$.  We shall need the notion of a {\em Lie subgroupoid} of a Lie
groupoid, which is closely modeled on the definition in
\cite{MackenzieBook2}. Recall that if $M$ is a manifold with corners
and $L \subset M$ is a subset, we say that $L$ is a submanifold of $M$
if it is a {\em closed} subset, if it is a manifold with corners in
its own with for topology induced from $M$, and if the inclusion $L
\to M$ is smooth and has injective differential.

\begin{definition}\label{def.subgroupoid}
Let $\maG \tto M$ be a Lie groupoid. A Lie groupoid $\maH \tto L$ is a
{\em Lie subgroupoid} of $\maG$ if $L$ is a submanifold of $M$ and
$\maH$ is a submanifold of $\maG$ with the groupoid structural maps
induced from $\maG$. (So $L$ and $\maH$ are closed subsets, according
to our conventions.)
\end{definition}

%\subsection{Lie algebroids}

Lie groupoids generalize Lie groups. By analogy, a Lie groupoid $\maG$
will have an associated infinitesimal object $A(\maG)$, the \dlp Lie
algebroid associated to to $\maG$.\drp\ To define it, let us first
recall the definition of a Lie algebroid. See
Pradines\rp\ \cite{Pradines} for the original definition and
Mackenzie\rp s books \cite{MackenzieBook2} a comprehensive
introduction to their general theory.

\begin{definition} \label{def.Lie.alg}
A {\em Lie algebroid} \m{A \to M} is a real vector bundle over a
Hausdorff manifold with corners \m{M} together with a {\em Lie algebra}
structure on \m{\Gamma(M;A)} (with bracket \m{[\ , \ ]}) and a vector
bundle map \m{\varrho: A \rightarrow TM}, called {\em anchor}, such
that the induced map \m{\varrho_* : \Gamma(M; A) \to \Gamma(M; TM)}
satisfies the following two conditions:
\begin{enumerate}[(i)]
\item \m{\varrho_*([X,Y]) = [\varrho_*(X),\varrho_*(Y)]} and
\item \m{[X, fY] = f[X,Y] + (\varrho_*(X) f)Y}, for all \m{X, Y \in
  \Gamma(M; A)} and \m{f \in \CI(M)}.
\end{enumerate}
\end{definition}

Morphisms of Lie algebroids are tricky to define in general (see for
instance 4.3.1 \cite{MackenzieBook2}), but we will need only special
cases. The isomorphisms are easy.  Two groupoids $A_i \to M_i$ are
{\em isomorphic} if there exists a vector bundle isomorphism $\phi :
A_1 \to A_2$ that preserves the corresponding Lie brackets.  If $M_1 =
M_2 = M$, we will consider morphisms {\em over $M$}.  (Often, however,
this \dlp over $M$\drp\ will be omitted.) Unless explicitly stated otherwise,
an isomorphism of two Lie algebroids will induced the {\em identity}
on the base, with the exception when this isomorphism comes from the 
action of a given Lie group. The same convention applies to the 
isomorphisms of Lie groupoids.

\begin{definition}\label{def.morphism}
 Let $A_i \to M$ be two Lie algebroids. A {\em morphism over $M$} of
 $A_1$ to $A_2$ is a vector bundle morphism $\phi : A_1 \to A_2$ that
 induces the identity over $M$ and is compatible with the anchor maps
 and the Lie brackets.
\end{definition}

More precisely, the map $\phi$ of this definition satisfies
$\varrho(\phi(X)) = \varrho(X)$ and $\varrho([X, Y]) = [\varrho(X),
  \varrho(Y)]$ for all sections $X$ and $Y$ of $A_1$.  See 3.3.1 of
\cite{MackenzieBook2}.

The following simple remark will be useful in the proof of Theorem
\ref{thm.tame2}.

\begin{lemma}\label{lemma.rem.prod}
 Let $A \to M$ be a Lie algebroid and $f \in \CI(M)$ be such that
 $\{f=0\}$ has an empty interior. Then $f \Gamma(M; A) \subset
 \Gamma(M; A)$ is a finitely generated, projective module and a Lie
 subalgebra. Thus there exists a Lie algebroid, denoted $fA$, such
 that $\Gamma(fA) := \Gamma(M; fA) \simeq f\Gamma(A)$.
\end{lemma}

\begin{proof} The proof of the Lemma relies on two simple calculations,
which nevertheless will be useful in what follows.  Let $X, Y \in
\Gamma(A) := \Gamma(M; A)$. We have
 \begin{equation}\label{eq.with.f}
  [fX, fY] \, = \, f X(f)Y - f Y(f)X + f\sp{2}[X, Y] \, \in \, \Gamma(fA) \,.
 \end{equation}
 The proof is complete.
\end{proof}

Recall the following definition (see \cite{MackenzieBook2, Rinehart}).

\begin{definition}\label{def.LieRinehart}
Let $R$ be a commutative associative unital real algebra and let
$\mfkg$ be a Lie algebra and an $R$-module such that $\mfkg$ acts by
derivations on $R$ and the Lie bracket satisfies
\begin{equation*}
 [X, rY] \, = \, r[X, Y] + X(r)Y\,, \quad \mbox{ for all }\ r \in R
 \ \mbox{ and }\ X, Y \in \mfkg\,.
\end{equation*}
Then we say that $\mfkg$ is an {\em $R$-Lie-Rinehart algebra}.
\end{definition}

Let $M$ be a compact manifold with corners.  We thus see that the
category of Lie algebroids with base $M$ is equivalent to the category
of finitely-generated, projective $\CI(M)$-Lie-Rinehart algebras, by
the Serre-Swan Theorem, Theorem \ref{SS}. It is useful in Analysis to
think of Lie algebroids as comming from Lie-Rinehart algebras.

We now recall some basic constructions involving Lie algebroids. See
\cite{MackenzieBook2} for more details.  For further
reference, let us introduce here the {\em isotropy} of a Lie
algebroid.

\begin{definition}\label{def.isotropy}
Let \m{\varrho : A \to TM} be a Lie algebroid on \m{M} with anchor
\m{\varrho}. Then the kernel \m{\ker (\varrho_x : A_x \to T_x M)} of
the anchor is the {\em isotropy} of \m{A} at \m{x \in M}.
\end{definition}

The isotropy at any point can be shown to be a Lie algebra.

\subsection{Direct products and pull-backs of Lie algebroids}
For the purpose of proving Theorems \ref{thm.tame} and
\ref{theorem.blow-up} below, we need a good understanding of thick
pull-back Lie algebroids and of their relation to vector
pull-backs. We thus recall the definition of the thick pull-back of a
Lie algebroid and of the direct product of two Lie algebroids. More
details can be found in \cite{MackenzieBook2}, however, we use a
simplified approach that is enough for our purposes. We therefore
adapt accordingly our notation and terminology. For instance, we
shall use the term \dlp thick pull-back of Lie algebroids\drp\ (as in
\cite{aln1}) in order to avoid confusion with the ordinary
(i.e. vector bundle) pull-back, which will also play a role. For
example, vector pull-backs appear in the next lemma, Lemma
\ref{lemma.prod1}, which states that a constant family of Lie
algebroids defines a new Lie algebroid. We first make the following
observations.

\begin{lemma}\label{lemma.prod1} 
Let $A_2 \to M_2$ be a vector bundle and $M_1$ be another manifold.
Let $A := p_2\sp{*}(A_2)$ be the vector bundle pull-back of $A_2$ to
the product $M_1 \times M_2$ via projection $p_2 : M_1 \times M_2 \to
M_2$. If $A_2 \to M_2$ is a Lie algebroid, then $A \to M_1 \times M_2$
is also Lie algebroid with $[f \otimes X, g \otimes Y] = fg \otimes [X
  , Y]$ for all $f, g \in \CI(M_1)$ and $X, Y \in \Gamma(A_2)$.
\end{lemma}

\begin{proof}
 This follows from definitions.
\end{proof}

\begin{remark}\label{rem.LR1}
A slight generalization of Lemma \ref{lemma.prod1} would be that if
$\mfk g$ is an $R$-Lie-Rinehart algebra and $R_1$ is another ring,
then $R \otimes \mfk g$ (tensor product over the real numbers) is an
$R_1 \otimes R$-Lie-Rinehart algebra, except that, in our case, we are
really considering also completions (of $R_1 \otimes R$ and of $R_1
\otimes \mfk g$) with respect to the natural topologies.
\end{remark}

We now make the Lie algebroid structure in Lemma \ref{lemma.prod1} more
explicit.

\begin{remark}\label{rem.LR.explicit}
 Let us identify $\Gamma(M_1 \times M_2; A) \simeq \CI(M_2 ;
 \Gamma(M_1; A_1))$. Then the Lie bracket on the space of sections of
 the $A \to M_1 \times M_2$ of Lemma \ref{lemma.prod1} is given by
\begin{equation*}
  [X, Y](m) \, := \, [X(m), Y(m)] \,,
\end{equation*}  
where $m \in M$ and $X, Y \in \Gamma(M_1 \times M_2; A) \simeq \CI(M_2
; \Gamma(M_1; A_1))$, so that the evaluations $X(m), Y(m) \in
\Gamma(M_1; A)$ are defined.  The anchor is
\begin{equation*}
 \varrho : A \, \to \, p_1\sp{*}(TM_1) = TM_1 \times M_2 \subset T(M_1
 \times M_2) \,.
\end{equation*}
\end{remark}

We now introduce products of Lie algebroids \cite{MackenzieBook2}
(our notation is slightly different from the one in that book).

\begin{corollary}\label{cor.prod1}
 Let $A_i \to M_i$, $i=1, 2$, be Lie algebroids and let
 $p_1\sp{*}(A_1)$ and $p_2\sp{*}(A_2)$ be their vector bundle
 pull-backs to $M_1 \times M_2$ (introduced in Lemma
 \ref{lemma.prod1}) with their natural Lie algebroid structures.  Then
\begin{equation*}
  A_1 \boxtimes A_2 \ede p_1\sp{*}(A_1) \oplus p_2\sp{*}(A_2) \simeq
  A_1 \times A_2 \to M_1 \times M_2
\end{equation*}
 has a natural Lie algebroid structure $A_1 \boxtimes A_2 \to M_1
 \times M_2$ such that $\Gamma(M_1; A_1)$ and $\Gamma(M_2; A_2)$
 commute in $\Gamma(M_1 \times M_2; A_1 \boxtimes A_2)$.  We notice
 that $\Gamma(M_1 \times M_2; p_i\sp{*}(A_i))$ is thus a sub Lie
 algebra of $\Gamma(M_1 \times M_2; A_1 \boxtimes A_2)$, $i = 1, 2$.
\end{corollary}

The Lie algebroid $A_1 \boxtimes A_2$ just defined is called the {\em
  direct product Lie algebroid} (see, for instance,
\cite{MackenzieBook2}) and is thus isomorphic, as a vector bundle, to
the product $A_1 \times A_2 \to M_1 \times M_2$. We shall need the
following important related construction.

\begin{definition}\label{def.thick-pb}
 Let $A \to L$ be a Lie algebroid over $L$ with anchor $\varrho: A \to
 TL$. Let $f : M \to L$ be a smooth map and define as in \cite[pages
   202--203]{HigginsMackenzie1}
 \begin{equation*}
  A \oplus_{TL} TM \, := \, \{\, (\xi, X) \in A \times TM, \,
  \varrho(\xi) = df(X) \in TL \, \}\,.
 \end{equation*}
Assume $A \oplus_{TL} TM$ defines a smooth vector bundle over $M$.
Then we define the {\em thick pull-back} Lie algebroid of $A$ by $f$
by $f\pullback (A) := A \oplus_{TL} TM$.
\end{definition}

As we will see shortly, it is easy to see that if $f$ is a 
tame submersion of manifolds with corners, then $f\pullback(A)$ is
defined. We shall use Lemma \ref{lemma.corners}(ii) to reduce to the
case of products, which we treat first.
% % 

\begin{lemma}\label{lemma.res.pb}
Let $A \to L$ be a Lie algebroid over a manifold with corners $L$ and
let $Y$ be a smooth manifold. If $f$ denotes the projection $L \times
Y \to L$, then
\begin{equation*} %\label{eq.product}
 f\pullback(A) \, \simeq \, A \boxtimes TY \, \simeq \, f\sp{*}(A)
 \oplus (L \times TY) \,,
\end{equation*}
the first isomorphism being an isomorphism of Lie algebroids and the
second isomorphism being simply an isomorphism of vector bundles.
\end{lemma}

\begin{proof}
  The result then follows from Definition \ref{def.thick-pb} and
  Corollary \ref{cor.prod1}.
\end{proof}

Thus, in general, the Lie algebroid pull-back (or thick pull-back)
$f\pullback A$ will be non-isomorphic to the vector bundle pull-back
$f\sp{*}(A)$. The following was stated in the smooth case in
\cite{MackenzieBook2}.

\begin{proposition}\label{prop.res.pb}
 Let $f : M \to L$ be a surjective tame submersion of manifolds with
 corners and $A \to L$ be a Lie algebroid. Then the thick pull-back
 $f\pullback(A)$ is defined (that is, it is a Lie algebroid).  Let
 $T_{vert}(f) := \ker(f_* : TM \to TL)$, then $T_{vert}(f) \subset
 f\pullback(A)$ is an inclusion of Lie algebroids and $A \simeq \big (
 f \pullback(A)/\ker(f_*) \big)\vert_{L}$ as vector bundles.
\end{proposition}

\begin{proof}
 This is a local problem, so the result follows from Lemma
 \ref{lemma.res.pb}.
\end{proof}

Let us recall now the definition of the Lie algebroid $A(\maG)$
associated to a Lie groupoid $\maG$, due to Pradine \cite{Pradines}.
Let $d, r : \maG \tto M$ be a Lie groupoid, then we let
\begin{equation*}
  A(\maG) \, := \, \ker(d_* : T\maG \to TM)\vert_{M} \,,
\end{equation*}
that is, $A(\maG)$ is the restriction to the units of the kernel of
the differential of the domain map $d$. The sections of $A(\maG)$
identify with the space of $d$-horizontal, right invariant vector
fields on $\maG$ (that is, vector fields on $\maG$ that are tangent to
the submanifolds $\maG_x := d\sp{-1}(x)$ and are invariant with
respect to the natural action of $\maG$ by right translations).  The
groupoid $\maG$ acts by right translations on $\maG_x$ in the sense
that if $\gamma \in \maG$ has $r(\gamma) = x$ and $d(\gamma) = y$,
then the map $\maG_x \ni h \to h \gamma \in \maG_y$ is a
diffeomorphism.  In particular, the space of sections of $A(\maG) \to
M$ has a natural Lie bracket that makes it into a Lie algebroid.

\begin{definition}\label{def.AG}
 Let $\maG \tto M$ be a Lie groupoid, then the Lie algebroid~$A(\maG)$
 is called the {\em Lie algebroid associated to $\maG$.}
\end{definition}

\section{Constructions with Lie groupoids}

We now introduce some basic constructions using Lie groupoids.

\subsection{Basic examples of groupoids}

We continue with various examples of constructions of Lie groupoids
and Lie algebroids that will be needed in what follows.

We begin with the three basic examples. Most of these examples are
extensions to the smooth category of some examples from the locally
compact category.  We will not treat the locally compact category
separately, however.

\begin{example}
\label{ex.Lie.group}
Any {\em Lie group} $G$ is a Lie groupoid with associated Lie
algebroid $A(G) = Lie(G)$, the Lie algebra of $G$. Let us assume $G$
unimodular, for simplicity, then the product on $\CIc(\maG) =
\CIc(G)$, Equation \ref{eq.def.convolution}, is simply the convolution
product with respect to the Haar measure.
\end{example}

At the other end of the spectrum, we have the following example.

\begin{example}\label{ex.space}
 Let $M$ be a manifold with corners and $\maG\sp{(1)} = \maG\sp{(0)} =
 M$, so the groupoid of this example contains only units. We shall
 call a groupoid with these properties a {\em space}.  We have $A(M) =
 M \times \{0\}$, the zero vector bundle over $M$ The product on
 $\CIc(\maG) = \CIc(M)$ is nothing but the pointwise product of two
 functions.
\end{example}

We thus see that the category of Lie groupoids contains the
subcategories of Lie groups and of manifolds (possibly with
corners). The last basic example is that of a product.

\begin{example}
 \label{ex.product}
 Let $\maG_i \tto M_i$, $i = 1, 2$, be two Lie groupoids. Then $\maG_1
 \times \maG_2$ is a Lie groupoid with units $M_1 \times M_2$.  We
 have $A(\maG_1 \times \maG_2) \simeq A(\maG_1) \boxtimes A(\maG_2)$,
 by Proposition 4.3.10 in \cite{MackenzieBook2}.
\end{example}

We shall need the following more specific classes of Lie groupoids.
The goal is to build more and more general examples that will lead us
our desired desingularization procedure. We proceed by small steps,
mainly due to the complicated nature of this construction, but also
because particular or intermediate cases of this construction are
needed on their own.  The following example is crucial in what
follows, since it will be used in the definition of the
desingularization groupoid.

\begin{example}\label{ex.BLG}
 Let $G$ be a Lie group with automorphism group $\Aut(G)$ and let $P$
 be a principal $\Aut(G)$-bundle. Then the associated fiber bundle
 $\maG := P \times_{\Aut(G)} G$ with fiber $G$ is a Lie groupoid
 called a {\em Lie group bundle} or a {\em bundle of Lie groups}. We
 have $d = r$ and $A(\maG) \simeq P \times_{\Aut(G)} Lie(G)$ in this
 example. We shall be concerned with this example especially in the
 following two particular situations. Let $\pi : E \to M$ be a smooth
 {\em real vector bundle} over a manifold with corners. Then each
 fiber $E_m := \pi\sp{-1}(m)$ is a commutative Lie group, and hence
 $E$ is a Lie groupoid with the corresponding Lie group bundle
 structure. The following frequently used example is obtained as
 follows. Let $\RR_{+}\sp{*} = (0, \infty)$ act on the fibers of the
 vector bundle $\pi : E \to M$ by dilation.  This yields, for each $m
 \in M$, the semi-direct product $G_m := E_m \rtimes
 \RR_{+}\sp{*}$. Then $\maG := \cup G_m$ is a Lie group bundle, and
 hence has a natural Lie groupoid structure.  Typically, we will have
 $E = A(\maH)$, the Lie algebroid of some Lie groupoid $\maH$, in
 which case these constructions appear in the definitions of the
 adiabatic groupoid and of the edge modification, and hence in the
 definition of the desingularization of a Lie groupoid. Equation
 \eqref{eq.def.convolution} becomes the fiberwise convolution.
\end{example}

\begin{example}
\label{ex.pair}
Let $M$ be a smooth manifold (thus $M$ {\em does not have
  corners}). Then we define the {\em pair groupoid} of $M$ as $\maG :=
M \times M$, a groupoid with units $M$ and with $d$ the second
projection, $r$ the first projection, and $(m_1, m_2)(m_2, m_3) =
(m_1, m_3)$. We have $A(M \times M) = TM$, with anchor map the
identity map.  A related example is that of $\maP M$, the {\em path
  groupoid} of $M$, defined as the set of fixed end point homotopy
classes of paths in $M$. It has the same Lie algebroid as the pair
groupoid: $A(\maP M) = TM$, but it leads to differential operators
with completely different properties (and hence to a different
analysis). See \cite{Gualtieri} for a description of all groupoids
integrating $TM$.
\end{example}

\begin{remark}
 The product on $\CIc(\maG) = \CIc(M \times M)$ is, in the case of the
 pair groupoid, simply the product of integral kernels. Let us fix a
 metric on $A(M \times M) = TM$, and hence a measure on $M$. Then
 Equation \eqref{eq.def.convolution} becomes
 \begin{equation}
  \phi \ast \psi (x, z) \, = \, \int_{M} \phi(x, y) \psi(y, z) dy\,.
 \end{equation}
 This is the reason why the pair groupoids are so basic in our 
 considerations.
\end{remark}

We need to recall the concept of a {\em morphism} of two groupoids,
because we want equivariance properties of our constructions.

\begin{definition}\label{def.groupoid.m}
 Let $\maG \tto M$ and $\maH \tto L$ be two groupoids. A {\em
   morphism} $\phi : \maH \to \maG$ is a functor of the corresponding
 categories.
\end{definition}

More concretely, given a morphism $\phi : \maH \to \maG$, it is
required to satisfy $\phi(gh) = \phi(g)\phi(h)$.  Then there will also
exists a map $L \to M$, usually also denoted by $\phi$, such that
$d(\phi(g)) = \phi(d(g))$, $r(\phi(g)) = \phi(r(g))$, and $\phi(u(x))
= u(\phi(x))$.

If $\maG \tto M$ and $\maH \tto L$ are Lie groupoids and the groupoid
morphism $\phi : \maH \to \maG$ is smooth, we shall say that $\phi$ is
a {\em Lie groupoid morphism}. If $\Gamma$ is a Lie group and $\maG
\tto M $ is a Lie groupoid, we shall say that $\Gamma$ {\em acts} on
$\maG$ if there exists a smooth map $\alpha : \Gamma \times \maG \to
\maG$ such that, for each $\gamma \in \Gamma$, the induced map
$\alpha_{\gamma} :\maG \ni g \to \alpha(\gamma, g) \in \maG$ is a Lie
groupoid morphism and $\alpha_{\gamma} \alpha_{\delta} =
\alpha_{\gamma \delta}$.

We now recall the important construction of {\em fibered pull-back
  groupoids} \cite{HigginsMackenzie1, HigginsMackenzie2}.

\begin{example}\label{ex.pullback} Let again $M$ and $L$ be locally
compact spaces and $f : M \to L$ be a continuous map. Let $d, r : \maH
\tto L$ be a locally compact groupoid (so $L$ is the set of units of
$\maH$), the {\em fibered pull-back groupoid} is then
\begin{equation*}
  f\pullback (\maH) \ede \{\, (m, g, m\rp) \in M \times
 \maH \times M, f(m) = r(g),\, d(g) = f(m\rp) \, \} \,. 
\end{equation*}
It is a groupoid with units $M$ and with 
%structural maps 
$d(m, g, m\rp) = m\rp$, $r(m, g, m\rp) = m$, and product $(m, g, m\rp)
(m\rp, g\rp, m\drp) = (m, g g\rp, m\drp)$.  We shall also sometimes
write $M \times_f \maH \times_f M = f\pullback (\maH)$ for the fibered
pull-back groupoid.  We shall use this construction in the case when
$f$ is a tame submersion of manifolds with corners and $\maH$ is a Lie
groupoid.  Then $\maG$ is a Lie groupoid (the fibered pull-back Lie
groupoid).  Indeed, to see that $d$ is a tame submersion, if is enough
to write that $f$ is locally a product, see Lemma
\ref{lemma.corners}(ii).  It is a subgroupoid of the product $M \times
M \times \maH$ of the pair groupoid $M \times M$ and $\maH$. Also by
Proposition 4.3.11 in \cite{MackenzieBook2}, we have
\begin{equation}
 A \big ( f\pullback (\maH) \big) \, \simeq \, f\pullback \big(
 A(\maH) \big) \,
\end{equation}
(see Definition~\ref{def.thick-pb}). Thus the Lie algebroid of the
fibered pull-back groupoid $f\pullback(\maH)$ is the thick pull-back
Lie algebroid $f\pullback \big( A(\maH) \big)$ and hence it contains
as a Lie algebroid the space $\ker (df)$ of $f$-vertical tangent
vector fields on $M$.  We note that if a Lie group $\Gamma$ acts
(smoothly by groupoid automorphisms) on $\maH \tto L$ and if the map
$f : M \to L$ is $\Gamma$-equivariant, then $\Gamma$ will act on
$f\pullback(\maH)$.
\end{example}

\subsection{Adiabatic groupoids and the edge-modification}
Our desingularization uses in an essential way {\em adiabatic groupoids.}
In this subsection, we shall thus recall in detail the construction of
the adiabatic groupoid, as well as some related constructions
\cite{ConnesBook, DebordSkandalis, NWX}.  For the purpose of further
applications, we stress the smooth action of a Lie group $\Gamma$ (by
Lie groupoid automorphisms) and thus the functoriality of our
constructions. We shall use the following standard notation.

\begin{notation}\label{not.res}
Let $d, r : \maG \tto M$ be a groupoid and $A, B \subset M$, then we
denote $\maG_A := d\sp{-1}(A)$ and $\maG_A\sp{B} := r\sp{-1}(B) \cap
d\sp{-1}(A)$. We also write $\maG_x := d\sp{-1}(x)$.
\end{notation}

In particular, $\maG_A\sp{A}$ is a groupoid with units $A$, called the
{\em reduction of $\maG$ to $A$}. In general, it will not be a Lie
groupoid even if $\maG$ is a Lie groupoid.  If $A \subset M$ is
$\maG$-invariant, meaning that $\maG_A\sp{A} = \maG_A = \maG\sp{A} :=
r\sp{-1}(A)$, then $\maG_A$ will be a groupoid, called the {\em
  restriction} of $\maG$ to (the invariant subset) $A$.

Let $\maG$ be a Lie groupoid with units $M$ and Lie algebroid $A :=
A(\maG) \to M$. The {\em adiabatic groupoid $\maG_{ad}$ associated to
  $\maG$} will have units $M \times [0, \infty)$.  We shall define
  $\maG_{ad}$ in several steps: first we define its Lie algebroid,
  then we define it as a set, then we recall the unique smooth
  structure that yields the desired Lie algebroid, and, finally, we
  show that this construction is functorial and thus preserves group
  actions.

\subsubsection{The Lie algebroid of the adiabatic groupoid}.
 We first define a Lie algebroid $A_{ad} \to M \times [0, \infty)$
   that will turn out to be isomorphic to $A(\maG_{ad})$, as in
   \cite{NWX}. As vector bundles, we have
 \begin{equation*}
  A_{ad} \ede  A \times [0, \infty) \to M \times [0, \infty) \,.  
 \end{equation*}
That is, $A_{ad}$ is the vector bundle bundle pull-back of $A \to M$
to $M \times [0, \infty)$ via the canonical projection $\pi : M \times
  [0, \infty) \to M$.  To define the Lie algebra structure on the
    space of sections of $A_{ad}$, let $X(t)$ and $Y(t)$ be sections
    of $A_{ad}$, regarded as smooth functions $[0, \infty) \to
      \Gamma(M; A(\maG))$. Then
\begin{equation}\label{eq.times.t}
  [X, Y](t) \ede t [X(t), Y(t)] \,.
\end{equation}
Let us denote by $\pi\sp{*}(A)$ the Lie algebroid defined by the
vector bundle pull-back, as in Lemma \ref{lemma.prod1}. Thus we see
that $A_{ad} \simeq \pi\sp{*}(A)$ as {\em vector bundles} but {\bf
  not} as Lie algebroids. Nevertheless, we do have a natural Lie
algebroid morphism (over $M \times [0, \infty)$, not injective!)
\begin{equation}\label{eq.Aad}
 A_{ad} \, \simeq \, t\pi\sp{*}(A) \, \to \pi\sp{*}(A) \,,
\end{equation}
where the second Lie algebroid is defined by Lemma
\ref{lemma.rem.prod} and the isomorphism is by Equation
\eqref{eq.times.t}. The induced map identifies $\Gamma(A_{ad})$ with
$t \Gamma(\pi\sp{*}(A))$, however.

\subsubsection{The underlying groupoid of $\maG_{ad}$}\label{sssec.underlying}
We shall define the adiabatic groupoid $\maG_{ad}$ as the union of two
Lie groupoids, denoted $\maG_1$ and $\maG_2$, which we define first.
This will also define the groupoid structure on $\maG_{ad}$ (but not
the smooth structure yet!).  We let $\maG_1 := A(\maG) \times \{0\}$
with the Lie groupoid structure of a bundle of commutative Lie groups
$A(\maG) \times \{0\} \to M$.  (That is $\maG_1$ is simply a vector
bundle, regarded as a Lie groupoid.)  The groupoid $\maG_2$ is given
by $\maG_2 := \maG \times (0, \infty)$, with the product Lie groupoid
structure, where $(0, \infty)$ is regarded as a space (as in Example
\ref{ex.space}). As a set, we then define the adiabatic groupoid
$\maG_{ad}$ associated to $\maG$ as the {\em disjoint} union
 \begin{equation}
  \maG_{ad} \ede \maG_1 \sqcup \maG_2 \ede \bigl( A(\maG) \times \{0\}
  \bigr) \sqcup \bigl( \maG \times (0, \infty)\bigr) .
 \end{equation}
We endow $\maG_{ad}$ with the natural groupoid structure 
$d, r: \maG_{ad} \to M \times [0, \infty)$, where $d$ and $r$ restrict to
each of $\maG_1$ and $\maG_2$ to the corresponding domain and range maps,
respectively. 
 
\subsubsection{The Lie groupoid structure on $\maG_{ad}$}
We endow $\maG_{ad} := \maG_1 \cup \maG_2$ with the unique smooth
structure that makes it a Lie groupoid with Lie algebroid $A_{ad}$, as
in \cite{NistorJapan}. 
We proceed as in 
\cite{ConnesBook, DebordSkandalis} using a (real version of) the 
\dlp deformation to the normal cone\drp\ considered in those papers.
Let us make that construction explicit in our
case. We thus choose connections $\nabla$ on all the
manifolds $\maG_x := d\sp{-1}(x)$, $x \in M$. As in \cite{NWX}, we can 
choose these connections such that the resulting
family of connections is invariant with respect to right
multiplication by elements in $\maG$. This gives rise to a smooth map
$\exp_{\nabla} : A = A(\maG) \to \maG$ that maps the zero section of
$A(\maG)$ to the set of units of $\maG$. There exists a neighborhood
$U$ of the zero section of $A(\maG)$ on which $\exp_{\nabla}$ is a
diffeomorphism onto its image. Let us define then $W = W_{U} \subset A
\times [0, \infty) = A_{ad}$ to be the set of pairs $(X, t) \in A
  \times [0, \infty)$ such that $tX \in U$ and define $\Phi : W \to
    \maG_{ad}$ by the formula
\begin{equation}\label{eq.exp.coordinates}
 \Phi(X, t) \ede 
 \begin{cases}
  \ (\exp_{\nabla}(tX), t) \in \maG \times (0, \infty) & \mbox{ if } 
  t>0 \\ 
%\mbox{ and } tX \in U\\
  \ \ (X, 0) \in A(\maG) \times \{0\} & \mbox{ if } t = 0\,.
 \end{cases}
\end{equation}
We define the smooth structure on $\maG_{ad}$ such that both the image
of $\Phi$ and the set $\maG \times (0, \infty)$ are open subsets of
$\maG_{ad}$ with the induced smooth structure coinciding with the
original one. The transition functions are smooth.  The fact that the
resulting smooth structure makes $\maG_{ad}$ a Lie groupoid follows
from the differentiability with respect to parameters (including
initial data) of solutions of ordinary differential equations. This
smooth structure does not depend on the choice of the connection
$\nabla$, since the choice of a different connection would just amount
to the conjugation with a local diffeomorphism $\psi$ of $\maG$ in a
neighborhood of the units.  By construction, the space of sections of
$A(\maG_{ad})$ identifies with $t \Gamma(\pi\sp{*}(A))$, and hence
$A(\maG_{ad}) \simeq A_{ad}$, as desired. (Note that by
\cite{NistorJapan, NWX}, it is known that there exists a unique Lie
groupoid structure on $\maG_{ad}$ such that the associated Lie
algebroid is $A_{ad}$.)

\subsubsection{Actions of compact Lie groups}
The following lemma states that the adiabatic construction is
compatible with Lie group actions. We state this as a lemma.

\begin{lemma}\label{lemma.ad.Gamma}
Let $\Gamma$ be a Lie group and assume that $\Gamma$ acts on $\maG
\tto M$, then $\Gamma$ acts on $\maG_{ad}$ as well.
\end{lemma}

\begin{proof}
We can see this as follows. We use the notation in
\ref{sssec.underlying}.  We obtain immediately an action of $\Gamma$
on both $\maG_1$ and $\maG_2$. To see that this extends to an action
on the adiabatic groupoid, we need to check the compatibility with the
coordinate map $\Phi$.  Let $V$ be a compact neighborhood of the
identity in $\Gamma$. We can choose an open neighborhood $U_1 \subset
U$ of the set of units of $\maG$ such that the action of $\Gamma$ on
$M$ maps $V \times U_1$ to $U$. Then $V \times W_{U_1}$ maps to $W_U$
and the resulting map is smooth by the invariance of the smooth
structure on $\maG_{ad}$ with respect to the choice of connection.
\end{proof}

\subsubsection{Extensions of the adiabatic groupoid construction}
We shall need two slight examples of generalizations of the adiabatic
groupoid construction. We shall use the reduction of a groupoid $G$ to
a subset $A$, which, we recall, is denoted $G_A\sp{A} := r\sp{-1}(A)
\cap d\sp{-1}(A)$.

\begin{example}\label{ex.pullback2} Let again $M$ and $L$ be manifolds 
with corners and $f : M \to L$ be a tame submersion of manifolds with
corners. Let $\maH \tto L$ be a Lie groupoid and adiabatic groupoid
$\maH_{ad} \tto L \times [0, \infty)$. Let $\maG := f \pullback (\maH)
  = M \times_{f} \maH \times_{f} M$ be the fibered pull-back
  groupoid. Then the {\em adiabatic groupoid of} $\maG$ (with respect
  to $f$) has units $M \times [0, \infty)$ and is defined by
\begin{equation*}
 \maG_{ad, f} \, := \, f_1 \pullback (\maH_{ad}) \,,
\end{equation*}
where $f_1 := (f, id) : M \times [0, \infty) \to L \times [0,
    \infty)$.  Unlike $\maG_{ad}$, the groupoid $\maG_{ad, f}$ will
    not be a bundle of Lie groups at time $0$, but will be the fibered
    pull-back of the Lie groupoid $A(\maH) \to L$, regarded as a
    bundle of Lie groups, by the map $f : M \to L$. More precisely,
    let $X := M \times \{0\}$, which is an invariant subset of the set
    of units of $M \times [0, \infty)$.  Then the restriction of
      $\maG_{ad, f}$ to $X$ satisfies
\begin{equation}
 (\maG_{ad, f})_{X} \, \simeq \, M \times_{f} A(\maH) \times_{f} M \,
  =: \, f\pullback \big( A(\maH) \big)\,.
\end{equation}
\end{example}

\begin{remark}
If $\maH = L \times L$, then $\maG = M \times M$ (so both $\maH$ and
$\maG$ are pair groupoids in this particular case) and $\maG_{ad, f}$
at time $0$ will be the fibered pull-back to $M$ of the Lie groupoid
$A(\maH) = TL \to L$.  In this particular case, the associated
differential operators on $\maG_{ad, f}$ model adiabatic limits, hence
the name of these groupoids (this explains the choice of the name \dlp
adiabatic groupoid\drp\ in \cite{NWX}).
\end{remark}

For the next example, we shall need an action of $\RR_{+}\sp{*}$ on
the last example.

\begin{remark}\label{rem.DS}
We use the same setting and notation as in Example \ref{ex.pullback2}
above and let $\RR_+\sp{*} = (0, \infty)$ act by dilations on the time
variable $[0, \infty)$. This action induces a family of automorphisms
  of $\maH_{ad}$, as in \cite{DebordSkandalis} if we let $s \in
  \RR_+\sp{*} = (0, \infty)$ act by $s\cdot (g, t) = (g, s\sp{-1}t)$
  on $(g, t) \in \maH \times (0, \infty) \subset \maH_{ad}$. Referring
  to Equation \eqref{eq.exp.coordinates} that defines a
  parametrization of a neighborhood of $A(\maH) \times \{0\} \subset
  \maH_{ad}$, we obtain
\begin{multline*}
 s \cdot \Phi(X, t) \ede s \cdot (\exp_{\nabla}(tX), t) \ede
 (\exp_{\nabla}(tX), s\sp{-1}t) \\
 = \, (\exp_{\nabla}(s\sp{-1}t sX), s\sp{-1}t) \, =: \, \Phi(sX,
 s\sp{-1}t) \,.
\end{multline*}
By setting $t = 0$ in this equation, we obtain by continuity that the
action of $s$ on $(X, 0)$ is $s(X, 0) = (sX, 0)$.
\end{remark}

We shall use this remark to obtain a (slight extension of a)
construction in \cite{DebordSkandalis}. Recall that if a Lie group
$\Gamma$ acts on a Lie groupoid $\maG \tto M$, then the {\em
  semi-direct product} \cite{MackenzieBook2, MoerdijkMrcun02} $\maG
\rtimes \Gamma$ is defined by $\maG \rtimes \Gamma = \maG \times
\Gamma$, as manifolds, and $\maG \rtimes \Gamma$ has units $M$ and
product $(g_1, \gamma_1)(g_2, \gamma_2) := (g_1 \gamma_1(g_2),
\gamma_1 \gamma_2)$.

\begin{example}\label{ex.DS}
We use the notation in Example \ref{ex.pullback2} and in Remark
\ref{rem.DS}.  In particular, we denote $f_1 := (f, id) : M \times [0,
  \infty) \to L \times [0, \infty)$.  The action of $\RR_{+}\sp{*}$
    commutes with $f_1$ and induces an action on $\maG_{ad, f} := f_1
    \pullback (\maH_{ad})$ and we let
\begin{equation*}
 \edge(M, f, \maH) \ede \maG_{ad, f} \rtimes \RR_+\sp{*}
 \ede f_1 \pullback (\maH_{ad}) \rtimes \RR_+\sp{*} \, =
 \, f_1 \pullback ( \maH_{ad} \rtimes \RR_+\sp{*} ) \,,
\end{equation*}
 be the associated semi-direct product groupoid. The space of units of
 $\edge(M, f, \maH)$ is $M \times [0, \infty)$. The groupoid
   $\maH_{ad} \rtimes \RR_+\sp{*}$ was introduced and studied in
   \cite{DebordSkandalis}.
\end{example}

Let us spell out in detail the structure of the groupoid $\edge(M, f,
\maH)$.
   
\begin{remark}\label{rem.str.edge}
To describe $\edge(M, f, \maH)$ as a set, we shall describe its
reductions to $M \times \{0\}$ and to $M \times (0, \infty)$ (that is,
we shall describe its reductions at time $t = 0$ and at time $t > 0$).
Let us endow $A(\maH)$ with the Lie groupoid structure of a
(commutative) bundle of Lie groups with units $L \times \{0\}$.  Then,
at time $t = 0$, $\edge(M, f, \maH)$ is the semi-direct product
$f\pullback (A(\maH)) \rtimes \RR_+\sp{*} \simeq f\pullback (A(\maH)
\rtimes \RR_+\sp{*})$, with $\RR_+\sp{*}$ acting by dilations on the
fibers of $A(\maH)$.  That is
\begin{equation}\label{eq.important}
 \edge(M, f, \maH)_{\{0\} \times M} \, \simeq \, (M \times_{f} A(\maH)
 \times_{f} M) \rtimes \RR_+\sp{*} \, = \, M \times_{f} ( A(\maH)
 \rtimes \RR_+\sp{*})\times_{f} M
\end{equation} 
Thus $\edge(M, f, \maH)_{M \times \{0\}}$ is the fibered pull-back to
$M \times \{0\}$ via $f$ of a bundle of solvable Lie groups on $L$.
On the other hand, the complement, that is, the reduction of $\edge(M,
f, \maH)$ to $M \times (0, \infty)$ is isomorphic to the product
groupoid $f \pullback (\maH) \times (0, \infty)\sp{2} , $ where the
first factor in the product is the fibered pull-back of $\maH$ to $M$
and the second factor is the pair groupoid of $(0, \infty)$.
\end{remark}

For the pair groupoid $\maG = M \times M$ with $M$ smooth, compact,
the example of the adiabatic groupoid is due to Connes
\cite{ConnesBook} and was studied in connection with the index theorem
for smooth, compact manifolds.  See \cite{ConnesBook, DebordSkandalis,
  NWX} for more details.

\subsubsection{The anisotropic construction\label{sssec.an}}  
We shall need also an {\em anisotropic variant} of the groupoid 
$\edge(M, f, \maH)$, which is easier to define, but currently less used 
in applications.
We continue to use the notation in Example \ref{ex.pullback2} and in Remark
\ref{rem.DS}. In particular, $f_1 := (f, id) : M \times [0,
  \infty) \to L \times [0, \infty)$. We now modify the definition of
  $\edge(M, f, \maH) := f_1 \pullback ( \maH_{ad} \rtimes \RR_+\sp{*} )$
by replacing $\maH_{ad}$ with the product groupoid $\maH \times [0, \infty)$.
Thus we define 
\begin{equation}\label{eq.def.anisotropic}
 \edge_{ni}(M, f, \maH) \ede f_1 \pullback \bigl 
 ( (\maH \times [0, \infty) )\rtimes \RR_+\sp{*} \bigr )\,.
\end{equation}
Let us consider the action of $\RR_+\sp{*} = (0, \infty)$ on
$[0, \infty)$ and denote by $\maT := 
[0, \infty) \rtimes \RR_{+}\sp{2}$ the corresponding groupoid semi-direct 
product. Then Equation~\eqref{eq.def.anisotropic} becomes 
\begin{equation}\label{eq.def.anisotropic2}
 \edge_{ni}(M, f, \maH) \ede f \pullback(\maH) \times \maT  \,.
\end{equation}

We have the following analog of Remark \ref{rem.str.edge}. 

\begin{remark}\label{rem.str.edge2}
 The natural morphism $\maH_{ad} \to \maH \times [0, \infty)$ that
   integrates the Lie algebroid morphism of Equation \eqref{eq.Aad},
   gives rise to a morphism
\begin{equation}\label{eq.def.Psi}
  \Psi : \edge(M, f, \maH)_{M \times (0, \infty)} \, \to \,
  \edge_{ni}(M, f, \maH)_{M \times (0, \infty)} \,. 
\end{equation}
We can describe this morphism and, at the same time, describe
$\edge_{ni}(M, f, \maH)$ as a set, by describing its restrictions to
$M = M \times \{0\}$ and to $M \times (0, \infty)$, using, at time $t
= 0$, the $\RR_{+}\sp{*}$-equivariant groupoid morphism $A(\maH) \ni
\xi \to u(d(\xi)) \in \maH$. Thus in the two equations below, the
composition is the restriction of $\Psi$ to $M$ and, respectively, to
$M \times (0, \infty)$:
\begin{equation}\label{eq.important2}
\begin{gathered}
 \edge(M, f, \maH)_{M} \, \simeq \, f\pullback(A(\maH)) \rtimes
 \RR_{+}\sp{*} \ \to \, f\pullback(\maH) \times \RR_{+}\sp{*} \,
 \simeq \, \edge_{ni}(M, f, \maH)_{M} \\
 \edge(M, f, \maH)_{M \times (0, \infty)} \, \simeq \,
 f\pullback(\maH) \times (0, \infty)\sp{2} \, \simeq \, \edge_{ni}(M,
 f, \maH)_{M \times (0, \infty)} \,.
\end{gathered}
\end{equation} 
\end{remark} 

The construction of the edge modifications is equivariant.

\begin{lemma}\label{lemma.edge.Gamma}
Let us assume with the same notation that a Lie group $\Gamma$ acts on
$\maH \tto L$ and that the tame submersion $f : M \to L$ is $\Gamma$
invariant. Then $\Gamma$ acts on $\edge(M, f, \maH)$ and on
$\edge_{ni}(M, f, \maH)$ in a way that is compatible with the
structure provided by Remarks \ref{rem.str.edge} and
\ref{rem.str.edge2}. In particular, the natural morphism $\edge(M, f,
\maH) \to \edge_{ni}(M, f, \maH)$ is $\Gamma$-equivariant.
\end{lemma}

\begin{proof} The action of $\Gamma$ on $\edge_{ni}(M, f, \maH)$ is simply the 
 product action comming from the isomorphism \eqref{eq.def.anisotropic2},
 with $\Gamma$ acting trivially on the groupoid $\maT$. We thus need 
 only consider the action of $\Gamma$ on $\edge(M, f, \maH)$.
 The group $\Gamma$ acts on $\maH_{ad}$ by Lemma \ref{lemma.ad.Gamma}.
 This action commutes with the action of $\RR_{+}\sp{*}$ by
 naturality. Hence we obtain an action of $\Gamma$ on $\maH_{ad}
 \rtimes \RR_{+}\sp{*}$. The result follows since $f : M \to L$ is
 $\Gamma$ invariant.
\end{proof}
  
\subsection{Glueing Lie groupoids\label{ssec.glueing}}
We shall need to \dlp glue\drp\ two Lie groupoids along an open subset
of the set of units above which they are isomorphic. This can be done
under certain conditions, and we review now this construction
following Theorem~3.4 in \cite{Gualtieri}.

Let $\maG_i \tto M_i$, $i = 1, 2$, be two Lie groupoids. (Thus the
sets of units $M_i$ are Hausdorff manifolds, possibly with corners.)
Let us assume that we are given open subsets $U_i \subset M_i$ such
that the reductions $(\maG_i)_{U_i}\sp{U_i}$, $i = 1, 2$, are
isomorphic via an isomorphism $\phi : (\maG_1)_{U_1}\sp{U_1} \to
(\maG_2)_{U_2}\sp{U_2}$ that covers a diffemorphism $U_1 \to U_2$,
also denoted by $\phi$. We define $M := M_1 \cup_{\phi} M_2$ as
follows.  Let us consider on the disjoint union $M_1 \sqcup M_2$ the
equivalence relation $\sim_{\phi}$ generated by $x \sim_{\phi}
\phi(x)$ if $x \in U_1$. Then $M_1 \cup_{\phi} M_2 := M_1 \sqcup
M_2/\sim_{\phi}$. We define similarly
\begin{equation}\label{eq.glueing}
 \maH \ede \maG_1 \cup_{\phi} \maG_2 \ede (\maG_1 \sqcup \maG_2)/\sim \,.
\end{equation} 
We shall denote by $U_1\sp{c} := M_1 \smallsetminus U_1$ the
complement of $U_1$ in $M_1$ and by $M_1 \cap \maG_1 U_1\sp{c} \maG_1$
the $\maG_1$-orbit of $U_1\sp{c}$ in $M_1$.  We shall use a similar
notation for $\maG_2$.

\begin{proposition}\label{prop.glueing}
 Let us assume that the set $\phi(U_1 \cap \maG_1 U_1\sp{c} \maG_1)$
 does not intersect $U_2 \cap \maG_2 U_2\sp{c} \maG_2$ and that $M :=
 M_1 \cup_{\phi} M_2$ is a Hausdorff manifold (possibly with corners).
 Then the set $\maH$ of Equation \eqref{eq.glueing} has a natural Lie
 groupoid structure with units $M$. We have $\maG_i \simeq
 (\maH)_{M_i}\sp{M_i}$.
\end{proposition}

\begin{proof}
This is basically a consequence of the definitions.  We define the
domain map $d : \maH \to M$ by restriction to each of the groupoids
$\maG_i$, which is possible since $h \sim h\rp$ implies $d(h) \sim
d(h\rp)$. We proceed similarly to define the range map $r$.

Let us now identify $\maG_i$ with subsets of $\maH$.  Hence now $U_1 =
U_2$ and $\phi$ is the identity. To define the multiplication for $g_j
\in \maH$, $j = 1, 2$, just note that the hypothesis ensures that, if
$g_j\in\maG_j$, for $j=1,2$, with $d(g_1) = r(g_2) \in M$, then, first
of all, $x := d(g_1) = r(g_2) \in M_1 \cap M_2 = U_1 = U_2$.  Next,
either $r(g_1)\in U_1$ or $d(g_2)\in U_2 = \phi(U_1) = U_1$, because
otherwise
\begin{equation*}
  d(g_1) \, = \, r(g_2) \in U_1 \, \cap \, \maG_1 U_1\sp{c} \maG_1 \,
  \cap \,\maG_2 U_2\sp{c} \maG_2 \,,
\end{equation*}
which is in direct contradiction with the hypothesis.  This means
that, in fact, $g_j \in \maG_i$, for the same $i$, and we can define
the multiplication using the multiplication in $\maG_i$.
\end{proof}

One of the differences between our result, Proposition
\ref{prop.glueing}, and Theorem 3.4 in \cite{Gualtieri} is that we are
not starting with a Lie algebroid that needs to be integrated, thus we
do not have orbits that we could use. See however \cite{Gualtieri} for
a discussion of the gluing procedure in the framework of manifolds
(and many other useful results).

\section{Desingularization groupoids}

We now introduce our desingularization construction of a Lie groupoid
along a tame submanifold.

\subsection{A structure theorem near tame submanifolds} 
We have the following basic definition.

\begin{definition}\label{def.tame}
 Let $A \to M$ be a Lie algebroid over a manifold $M$.  Let $L \subset
 M$ be a submanifold of $M$ such that there exists a tubular
 neighborhood $U$ of $L$ in $M$ with projection map $\pi : U \to
 L$. We shall say that $L$ is an $A$-{\em tame submanifold} of $M$ if
 there exists also a Lie algebroid $B \to L$ such that the restriction
 of $A$ to $U$ is isomorphic to the thick pull-back Lie algebroid of
 $B$ to $U$ via $\pi$, that is,
\begin{equation}\label{def.tame.eq1}
 A \vert_{U} \, \simeq \, \pi\pullback(B) \,,
\end{equation}
via an isomorphism that is the identity on $U$. Both $M$ and $L$ are
allowed to have corners.
\end{definition}

\begin{remark}
We note that, by Proposition \ref{prop.res.pb}, we have that the Lie
algebroid $B$ of Definition \ref{def.tame} satisfies $B \simeq
(A/\ker(\pi_*))\vert_L$, and hence $B$ is determined up to an
isomorphism by $A$.
\end{remark}

Recall that if $\maG \tto M$ is a groupoid and $A \subset M$, then
$\maG_A\sp{A} := r\sp{-1}(A) \cap d\sp{-1}(A)$ is the {\em reduction
  of $\maG$ to $A$}. We shall use repeatedly the fact that, if $B
\subset A$, then $(\maG_A\sp{A})_{B}\sp{B} = \maG_{B}\sp{B}$.  Also,
recall that a topological space is called {\em simply-connected} if it
is path connected and its first homotopy group $\pi_1(X)$ is
trivial. A groupoid $\maG$ is called {\em $d$-simply connected} if the
fibers $\maG_x := d\sp{-1}(x)$ of the domain map are simply-connected.
Here is one of our main technical results that provides a canonical
form for a Lie groupoid in the neighborhood of a tame submanifold. All
the isomorphisms of Lie groupoids are smooth morphisms.

\begin{theorem}\label{thm.tame}
 Let $\maG \tto M$ be a Lie groupoid and let $L \subset M$ be an
 $A(\maG)$-tame submanifold of $M$.  Let $U \subset M$ be a tubular
 neighborhood of $L$ as in Definition \ref{def.tame}, with $\pi : U
 \to L \subset U$ the associated structural projection. Then the
 reduction groupoids $\maG_{L}\sp{L}$ and $\maG_U\sp{U}$ are Lie
 groupoids. Assume, furthermore, that the fibers of $\pi : U \to L$
 are simply-connected. Then there exists an isomorphism
\begin{equation*}
  \maG_{U}\sp{U} \, \simeq 
 \, \pi\pullback ( \maG_{L}\sp{L} ) \ede 
   U \times_{\pi} \maG_{L}\sp{L} \times_{\pi} U\,
\end{equation*}
of Lie groupoids that is the identity on the set of units $U$.
\end{theorem}

\begin{proof} 
First of all, we have that $\maG_L\sp{L}$ is a Lie groupoid by
\cite[Proposition 1.5.16]{MackenzieBook2} since the joint map $(d, r)
: \maG \to M \times M$ is transverse to $L \times L$. This
transversality property follows from the $A$-tameness of $L$, which in
turn implies that for every $g\in U$ we have that
\begin{equation*} 
  (d_\ast, r_\ast) (T_g\maG) \supset T_{d(g)}U \times T_{vert,
    r(g)}\pi.
\end{equation*}
Then, the fibered pair groupoid $\maH := U \times_\pi U = \{(u_1,
u_2), \pi(u_1) = \pi(u_2) \}$
is a Lie groupoid with Lie algebroid $T_{vert}\pi = \ker(d \pi)$.  The
assumption that the fibers of $\pi : U \to L$ are simply-connected
shows that $\maH$ is $d$-simply connected. Since $T_{vert}\pi$ is
contained in $A(\maG)\vert_U$ as a Lie subalgebroid, by the definition
of a $A(\maG)$-tame submanifold.  Proposition 3.4 of
\cite{MoerdijkMrcun02} (see also \cite{MackenzieBook2, NistorJapan})
gives there exists a morphism of Lie groupoids
\begin{equation}
  \Phi \, : \, \maH \ede U \times_\pi U  \, \to \, \maG_{U}\sp{U}
\end{equation}
that preserves the units, in the sense that $d(\Phi(\gamma)) =
d(\gamma)$ and $r(\Phi(\gamma)) = r(\gamma)$.  This gives that $\Phi$
is injective.

% fibered pair gropoid is used here
Let $(u_1, u_2) \in U \times U$ with $\pi(u_1) = \pi(u_2)$. Then
$(u_1, u_2) \in \maH := U \times_\pi U$. In particular, $(\pi(u), u)
\in L \times_\pi U \in \maH$ for any $u \in U$. Let us denote by $g(u)
:= \Phi(\pi(u), u) \in \maH \subset \maG$, which defines a smooth map
$g : U \to \maG$, since $\Phi$ is smooth. Then, for any $\gamma \in
r\sp{-1}(U) \cap d\sp{-1}(U) =: \maG_U\sp{U}$ and any $u \in U$, we
have
\begin{equation*}
 \begin{gathered}
  d(g(u)) = d(\pi(u), u) = u \,, \quad r(g(u)) = r(\pi(u), u) = \pi(u)
  \in L \,,\\
  h_{\gamma} := g(r(\gamma)) \gamma g(d(\gamma))\sp{-1} \in
  \maG_L\sp{L}\, , \ \ \mbox{ and }\ \ \Psi(\gamma) : = (r(\gamma),
  h_{\gamma} , d(\gamma)) \in U \times_{\pi} \maG_{L}\sp{L}
  \times_{\pi} U\,.
 \end{gathered}
\end{equation*}
The map $\Psi$ is the desired isomorphism.
\end{proof}

\subsection{The edge modification}
We shall now use the structure theorem, Theorem \ref{thm.tame} to
provide a desingularization of a Lie groupoid in the neighborhood of a
tame submanifold of its set of units. We need, however, to first
discuss the (real) blow-up of a tame submanifold. We use the standard
approach, see for example \cite{schroedinger, kottkeMelrose}.

\begin{notation}\label{not.submanifold}{\normalfont
 In what follows, $L$ will be a {\em tame} submanifold a manifold with
 corners (to be specified each time), that is, a submanifold with the
 property that it has a tubular neighborhood $U$ with structural
 projection $\pi : U \to L$.  We let $S := \pa
 U$. We shall denote by $NL$ the normal bundle of $L$ in $M$. We
 assume that $U$ identifies with the set of unit vectors in $NL$ using
 some connection. In particular, $S \simeq SNL$, the set of unit
 vectors in $NL$, and $U \smallsetminus L \simeq S \times (0, 1)$.}
\end{notation}

We now recall the definition of the real blow-up of a manifold with
respect to a tame submanifold. We use the notation introduced in
\ref{not.submanifold}.  Let us assume that $L$ is a tame submanifold
of a manifold with corners $M$.  Informally, the {\em real blow-up}
or, simply, the {\em blow-up} of $M$ along $L$ is the manifold with
corners obtained by removing $L$ from $M$ and gluing back $S \simeq
SNL$ in a compatible way.

\begin{definition}\label{def.blow-up}
Let $L$ be a tame submanifold of a manifold with corners $M$.  We use
the notation introduced in \ref{not.submanifold} and we let $\phi$ be
the diffeomorphism $U \smallsetminus L \simeq S \times (0, 1)$.  Then
the {\em real blow-up} of $M$ along $L$, denoted $[M:L]$ is defined by
glueing $M \smallsetminus L$ and $S \times [0, 1)$ using $\phi$, that
  is
\begin{equation}\label{eq.blow-up}
 [M:L] \ede (M \smallsetminus L) \cup_{\phi} (S \times [0, 1)) \, = \,
   \bigl ( (M \smallsetminus L) \sqcup S \times [0, 1) \bigr )/\sim
     \,,
\end{equation}
where $\sim$ is the equivalence relation defined by $\phi(x) \sim x$,
as in the Subsection \ref{ssec.glueing}.
\end{definition}

\begin{remark}\label{rem.kappa}
By construction, there exists an associated natural smooth map
\begin{equation*}
  \kappa : [M: L] \, \to \, M \,,
\end{equation*}
the {\em blow-down map}, which is uniquely determined by the
condition that it be continuous and it be the identity on $M \setminus
L$. For example,
\begin{equation}\label{ex.blow-up}
   [\RR\sp{n+k}: \{0\} \times \RR\sp{k}] \, \simeq \, S\sp{n-1} \times
   [0, \infty) \times \RR\sp{k}\, ,
\end{equation}
with $r \in [0, \infty)$ representing the distance to the submanifold
  $L = \{0\} \times \RR\sp{k}$ and $S\sp{p}$ denoting the sphere of
  dimension $p$ (the unit sphere in $\RR\sp{p+1}$). Locally, all
  blow-ups that we consider are of this form.  In this example,
  Equation \eqref{eq.blow-up}, the blow-down map is simply
  $\kappa(x\rp, r, x\drp) = (r x\rp, x\drp) \in \RR\sp{n} \times
  \RR\sp{k}$.
\end{remark}  
  
The definition of the blow-up in this paper is the one common in
Analysis \cite{schroedinger, BMNZ, grushin71, kottkeMelrose,
  mazzeo91}, however, it is {\em different} from the one in
\cite{Arone, Gualtieri, Polishchuk}, who replace $L$ with $PN =
SNL/\ZZ_2$, the projectivization of $SNL$, instead of $S := SNL$.
We are ready now to introduce the desingularization of a Lie groupoid
with respect to a tame submanifold in the particular case of a
suitable pull-back.

\begin{definition} \label{def.edge}
Let $\pi : E \to L$ be an orthogonal vector bundle (so $U \subset E$
is the set of vector of length $< 1$ and $S := \pa U \subset E$). The
various restrictions of $\pi$ will also be denoted by $\pi$. Let $\maH
\tto L$ be a Lie groupoid and $\maG := \pi\pullback (\maH) = U
\times_{\pi} \maH \times_{\pi} U$.  Then {\em the edge modification}
of $\maG$ is the fibered pull-back groupoid
\begin{equation*}
 \edge(S, \pi, \maH) \, := \, (S \times_{\pi} \maH_{ad} \times_{\pi}
 S) \rtimes \RR_{+}\sp{*} \, \simeq \, S \times_{\pi} (\maH_{ad}
 \rtimes \RR_{+}\sp{*}) \times_{\pi} S \,.
\end{equation*}
\end{definition}

\begin{remark}
The edge modification is thus a particular case for $f = \pi : M = S
\to L$ of the example \ref{ex.DS}. It is a Lie groupoid with units $S
\times [0, \infty)$. We extend in an obvious way the definition of the
  edge modification to groupoids isomorphic to groupoids of the form
  $\maG = \pi\pullback (\maH) = U \times_{\pi} \maH \times_{\pi} U$.
\end{remark}

It will be convenient to fix the following further notation.

\begin{notation}\label{not.blow} {\rm
In what follows, $\maG \tto M$ will denote a Lie groupoid and $L
\subset M$ will be an $A(\maG)$-tame submanifold. The sets $U$ and $S
:= \pa U$ have the same meaning as in \ref{not.submanifold}. In
particular, $\pi : U \to L$ is a tubular neighborhood of $L$, chosen
as in Definition \ref{def.tame}.  Using Theorem \ref{thm.tame}, we
obtain that the reduction $\maG_{U}\sp{U}$ is of the form
$\pi\pullback (\maH) := U \times_{\pi} \maH \times_{\pi} U$, and hence
its edge-modification $\edge(S, \pi, \maH)$ is defined. Let $M_1 = S
\times [0, 1)$, which is an open subset of the set $S \times [0,
    \infty)$ of units of $\edge(S, \pi, \maH)$. We shall denote by
    $\maG_1$ the reduction of $\edge(S, \pi, \maH)$ to $M_1$ and by
    $U_1 := U \smallsetminus L = S \times (0, 1) \subset
    M_1$. Similarly, $\maG_2$ will denote the reduction of the
    groupoid $\maG$ to $M \smallsetminus L$.}
\end{notation}

\begin{remark}\label{rem.glueing}
Using the notation and assumptions of Definition \ref{def.edge} and
the notation introduced in \ref{not.submanifold} and \ref{not.blow},
we have that the reduction of $\maG_1$ to $U_1$ (which by the
definition of $\maG_1$ the reduction of $\edge(S, \pi, \maH)$ to $U_1
:= U \smallsetminus L = S \times (0, 1) \subset M_1$) is isomorphic to
\begin{equation}
  (\maG_1)_{U_1}\sp{U_1} \, \simeq \, (\edge(S, \pi, \maH))_{U_1}\sp{U_1}
  \, \simeq \, 
  (S \times_{\pi} \maH \times_{\pi} S) \times (0, 1) \sp{2} \, \simeq \,
  U_1 \times_{\pi} \maH \times_{\pi} U_1 \,,
\end{equation}
where $(0, 1)\sp{2}$ is the pair groupoid. Since the reduction of
$\maG$ to $U$ is isomorphic to $U \times_{\pi} \maH \times_{\pi} U$,
it follows that the reduction of $\maG$ to $U_1$ is isomorphic to $U_1
\times_{\pi} \maH \times_{\pi} U_1$. Hence the reduction of $\maG_2$
to $U_1$ is also isomorphic to $U_1 \times_{\pi} \maH \times_{\pi}
U_1$.  We thus obtain an isomorphism of Lie groupoids
\begin{equation}\label{eq.def.phi}
 \phi : (\maG_1)_{U_1}\sp{U_1} \to (\maG_2)_{U_1}\sp{U_1} \, \simeq \,
 U_1 \times_{\pi} \maH \times_{\pi} U_1 \,.
\end{equation}
We are thus in position to glue the groupoids $\maG_1$ and $\maG_2$
along their isomorphic reductions to $U_1$, using Proposition
\ref{prop.glueing} (for $U_2 = U_1$).  We can now define the blow-up
of a groupoid with respect to a tame submanifold.
\end{remark}

\begin{definition}\label{def.blow-upG}
Let $L \subset M$ be an $A(\maG)$-tame manifold. Using the notation
just defined in Remark \ref{rem.glueing}, the result of glueing the
groupoids $\maG_1$ and $\maG_2$ along their isomorphic reductions to
$U_1 = S \times (0, 1)$ using Proposition \ref{prop.glueing} is
denoted $[[\maG: L]]$ and is called the {\em desingularization of
  $\maG$ along $L$}.
\end{definition}

\begin{remark}\label{rem.U1c}
 We note that the hypothesis of Proposition \ref{prop.glueing} are
 satisfied because $U_1\sp{c}$ is an invariant subset of $M_1$.
\end{remark}

\begin{remark}
 To summarize the construction of the desingularization, let us denote
 by $\phi$ the natural isomorphism of the following two groupoids:
 $\edge(S, \pi, \maH)_{U_1}\sp{U_1}$ (reduction to $U_1 \simeq S
 \times (0, 1)$) and $\maG_{U_1}\sp{U_1} =
 (\maG_2)_{U_1}\sp{U_1}$. Then
 \begin{equation}
  [[\maG:L]] \ede \edge(S, \pi, \maH)_{U_1}\sp{U_1} \cup_{\phi}
   \maG_{M \smallsetminus L}\sp{M \smallsetminus L} \, =: \, 
   \maG_1 \cup_{\phi} \maG_2 
  = (\maG_1 \cup_{\phi} \maG_2)/\sim\,.
 \end{equation}
\end{remark}

One should not confuse $[[\maG: L]]$ with $[\maG: L]$, the blow-up of
the manifold $\maG$ with respect to the submanifold $L$. Recall that
$\kappa : [M:L] \to M$ denotes the blow-down map, see Remark
\ref{rem.kappa}. It is uniquely determined by continuity and the
requirement that $\kappa$ be the identity on $M \smallsetminus L$. 

The following result is crucial in studying the desingularization
$[[\maG:L]]$. Recall that we endow $A(\maH) \to L$ with the Lie groupoid 
structure of a bundle of
 Lie groups, that $S \ede \kappa\sp{-1}(L) \, = \, [M:L] \smallsetminus (M
  \smallsetminus L)$ is a closed subset, and that $\pi : S \to L$ is the 
  natural (fiber bundle)
  projection.

\begin{proposition}\label{prop.blow-up}
The space of units of $[[\maG:L]]$ is $[M: L]$ and $S :=
\kappa\sp{-1}(L)$ is a $[[\maG:L]]$-invariant subset of $[M: L]$ 
with complement $[M:L]\smallsetminus S = M \smallsetminus L$, and
\begin{equation*}
 [[\maG: L]]_{S} \, \simeq \, \pi \pullback
 \bigl (A(\maH) \rtimes \RR_{+}\sp{*}) \ \ \mbox{ and } \ \
 [[\maG: L]]_{M \smallsetminus L} \, = \,
 \maG_{M \smallsetminus L}\sp{M \smallsetminus L} \, .
\end{equation*}
\end{proposition}

\begin{proof}
When we glue groupoids, we also glue their units, which gives that the
set of units of $[[\maG:L]]$ is indeed $M_1 \cup_{\phi} M_2 =:[M:L]$.
We have that $S \simeq S \times \{0\}$ is a closed, invariant subset
of the set $S \times [0, 1)$ of units of $\maG_1$, the reduction of
  $\edge(S, \pi, \maH)$ to $M_1$. Moreover, $(\maG_1)_S$ is the
  complement of the common part of the groupoids $\maG_1$ and $\maG_2$
  that are glued to yield $[[\maG:L]]$. Therefore $S$ is a
  $[[\maG:L]]$-invariant subset of $[M:L]$. (See also Remark
  \ref{rem.U1c}.) In particular,
\begin{equation*}
 [[\maG:L]] \, = \, (\maG_1)_S \, = \, \edge(S, \pi, \maH)_{S}\,.
\end{equation*}
The rest follows from the construction of $[[\maG: L]]$ and the
discussion in Example \ref{ex.DS}, Remark \ref{rem.str.edge}, and,
especially, Equation \eqref{eq.important}.
\end{proof}

Similar structures arise in other situations; see, for instance,
\cite{damakGeorgescu, krainer2014, georgescuNistor,
  Mantoiu1, mantoiuReine, nicolaRodinoBook, frascatti,
  seiler1999, schroheSymbol}. See also the discussion at the end 
of Example \ref{ex.BLG}. Proposition \ref{prop.blow-up} is important 
in Index theory and Spectral theory because it gives rise to exact 
sequences of algebras \cite{CuntzQuillen0, frascatti}.

These constructions extend to yield an anisotropic desingularization.

\begin{remark}
 Similarly, by considering the groupoid $\edge_{ni}(S, \pi, \maH)$
 instead of the groupoid $\edge(S, \pi, \maH)$, we obtain the {\em
   anisotropic desingularization} $[[\maG:L]]_{ni}$:
 \begin{equation}
  [[\maG:L]]_{ni} \ede \edge_{ni} (S, \pi, \maH)_{U_1}\sp{U_1}
  \cup_{\phi} \maG_{M \smallsetminus L}\sp{M \smallsetminus L} \,.
 \end{equation}
The anisotropic desingularization $[[\maG:L]]_{ni}$ will have the same
set of units as $[[\maG:L]]$, since they are obtained by gluing
groupoids with the same sets of units. From Equation
\eqref{eq.def.Psi}, we also obtain a natural morphism $\Psi :
      [[\maG:L]] \to [[\maG:L]]_{ni}$, that is the identity over the
      common set of units $[M:L]$.
\end{remark}

Proposition \ref{prop.blow-up} and its proof extend to the anisotropic case.

\begin{proposition}\label{prop.blow-up2}
The structure of $\maK := [[\maG:L]]_{ni}$ and of the natural morphism
$\Psi : [[\maG:L]] \to \maK := [[\maG:L]]_{ni}$ is as follows:
\begin{equation*}
 [[\maG: L]]_{S} \, \simeq \, \pi \pullback \bigl (A(\maH) \rtimes
 \RR_{+}\sp{*}) \, \to \, \pi \pullback(\maH \times \RR_{+}\sp{*}) \,
 \simeq \, \maK_{S}
\end{equation*}
and $\Psi = id : [[\maG: L]]_{M \smallsetminus L} \to \maK_{M
  \smallsetminus L} = \maG_{M \smallsetminus L}\sp{M \smallsetminus
  L}.$
\end{proposition}

The local structure of these constructions is discussed in Subsection
\ref{ssec.example}.  The desingularizations are compatible with Lie
group actions.

\begin{proposition}\label{prop.des.Gamma}
Let us assume that a Lie group $\Gamma$ acts on $M$ such that it
leaves invariant the tame submanifold with corners $L \subset M$.
Then $\Gamma$ acts on $[M:L]$ as well. If, moreover, $L$ is
$A(\maG)$-tame for some groupoid $\maG \tto M$ on which $\Gamma$ acts,
then we obtain that $\Gamma$ acts on $[[\maG:L]]$ and
$[[\maG:L]]_{ni}$ also.
\end{proposition}

\begin{proof}
The action on $[M:L]$ is obtained by the same argument as in the proof
of Lemma \ref{lemma.ad.Gamma} by considering a compact neighborhood of
the identity in $\Gamma$. We now show that $\Gamma$ acts on
$[[\maG:L]]$. Since $M \smallsetminus L$ is $\Gamma$-invariant,
$\Gamma$ will act on $\maG_2 := \maG_{M \smallsetminus L}\sp{M
  \smallsetminus L}$. By Lemma \ref{lemma.edge.Gamma}, $\Gamma$ acts
on $\edge(S, \pi, \maH)$.  These actions coincide on the common
domain, and hence $\Gamma$ acts on $[[\maG:L]]$.
\end{proof}

\subsection{The Lie algebroid of the desingularization}
We can now describe the Lie algebroid of the desingularization $[[
    \maG : L ]]$ of a \ssub\ Lie groupoid $\maG$ with respect to an
$A(\maG)$-tame submanifold $L \subset M$. Recall the definition of
$R$-Lie-Rinehart algebras \ref{def.LieRinehart} We have the following
extension of \cite[Theorem 3.10]{schroedinger} that was proved
originally for Lie manifolds.

\begin{notation}\label{not.desing}
 In the following, $A \to M$ will be a Lie algebroid and $L \subset M$
 will be an $A$-tame submanifold of $M$. Also, we shall denote by $r_L
 : M \to [0, \infty)$ a function that is $>0$ and smooth on $M
   \smallsetminus L$ and coincides with the distance to $L$ in a small
   neighborhood of $L$. We continue to denote by $[M:L]$ the blow-up
   of $M$ along $L$.
\end{notation}

We notice that the function $r_L$ lifts to a smooth function on
$[M:L]$, which is the main reason for introducing the blow-up $[M:L]$.

\begin{theorem} \label{thm.tame2}
 Let $\maW := \CI([M: L]) \otimes_{\CI(M)} r_L \Gamma(M; A)$, where we
 use the notation \ref{not.desing}. Then $\maW$ is a finitely
 generated, projective $\CI([M: L])$-module with the property that the
 given Lie bracket on $\CIc(M \smallsetminus L; A) \subset \maW$
 extends to $\maW$. Hence, there exists a Lie algebroid $[[A:L]] := B
 \to [M: L]$ such that $\Gamma([M:L]; B) \simeq \maW$.
\end{theorem}

The isomorphism $\Gamma([M:L]; B) \simeq \maW$ is an isomorphism of
vector bundles inducing the identity over $[M:L]$ and an isomorphism
of Lie algebras, hence it is an isomorphism of $\CI([M:
  L])$-Lie-Rinehart algebras.

\begin{proof}
The proof follows the lines of the proof of Theorem 3.10 in
\cite{schroedinger}, using the $A$-tameness of $L$ in order to
construct the Lie algebra structure on $\Gamma(M; A)$.  We include the
details for the benefit of the reader, taking also advantage of the
results in Subsection \ref{ssec.dLgLa}. In particular, we shall use
the local product structure of the thick pull-back of Lie algebroids,
Corollary \ref{cor.prod1} and Lemma \ref{lemma.res.pb}.

We have that $\Gamma(M; A)$ is a projective $\CI(M)$-module, hence
$r_L\Gamma(M; A)$ is a projective $\CI(M)$-module, and hence $\maW :=
\CI([M: L]) \otimes_{\CI(M)} r_L \Gamma(M; A)$ is a finitely
generated, projective $\CI([M: L])$-module. It remains to define the
Lie bracket on $\maW$. We shall prove fact more than that, namely, we
shall obtain in Equation \ref{eq.four} a local structure result for
$\maW$, which will be formalized in a few corollaries that follow the
proof.

We shall use the notation introduced in \ref{not.submanifold}. In
particular, $\pi: U \to L$, $L \subset U$ is the tubular neighborhood
used to define the thick pull-back of $A\vert_{L}$ to $U$. The problem
is local, so we may assume that $U = L \times \RR\sp{n}$.  Since $A$
is the thick pull-back of a Lie algebroid $A_1 \to L$ to $U$, we have
by Lemma \ref{lemma.res.pb} that
\begin{equation} \label{eq.product2}
 A\vert_{U}\, \simeq \, \pi \pullback(A_1) \, \simeq \, A_1 \boxtimes
 T\RR\sp{n} \, = \pi\sp{*}(A_1) \oplus (L \times T\RR\sp{n}) \, .
\end{equation}
We want to lift the sections of $A$ on $U$ to the blow-up $[U :
  L]$. This is, of course, possible for the sections of $\pi\sp{*}(A_1)
\to U$, but not for the sections of $L \times T\RR\sp{n} \to L$.  This
is why we need to multiply with the factor $r_L$.  We next use a
lifting result for vector fields from $\RR\sp{n}$ to
\begin{equation*}
 R_0 \ede [\RR\sp{n}:0] \, = \, S\sp{n-1} \times [0, \infty) \,.
\end{equation*}
Let $r(x) := |x|$ denote the distance to the origin in $\RR\sp{n}$.
We recall \cite{schroedinger} that a vector field  
$X \in r \Gamma(\RR\sp{n}, T\RR\sp{n})$ lifts to the blow-up $R_0$
and the resulting lift is tangent to the  boundary of
the blow-up (which, we recall, is $S\sp{n-1}$). Thus
\begin{multline}\label{eq.product3}
 \CI(R_0) \otimes_{\CI(\RR\sp{n})} r\Gamma(\RR\sp{n}, T\RR\sp{n}) \,
 \simeq \, \maV_b(R_0) \\ \, \simeq \, \Gamma(R_0; TS\sp{n-1} \boxtimes
 T\sp{b} [0, \infty) ) \, \simeq \, \Gamma(R_0; TS\sp{n-1} ) \, \oplus
   \, r \Gamma(R_0; T [0, \infty) ) \,,
\end{multline}
where $\maV_b(R_0)$ is as defined in \ref{not.Gamma}.
We may also assume that $r_L : M = U = L \times \RR\sp{n} \to [0, \infty)$
is given by $r_L(x, y) = r(y)$, again since the problem is local.  
\begin{equation*}
  M_1 \ede [M: L] \, = \, L \times [\RR\sp{n}: 0] \, = \, L \times
  S\sp{n-1} \times [0, \infty) \,.
\end{equation*} 
We now identify the spaces of sections of the vector bundles of 
interest using Equation \eqref{eq.product2}, the isomorphisms below being 
isomorphisms of $\CI(M_1)$-modules
\begin{multline}\label{eq.ref.mult}
 \maW \ede \CI(M_1) \otimes_{\CI(M)} r_L \Gamma(M; A)  \simeq 
 \CI(M_1) \otimes_{\CI(M)} r_L \Gamma(M; A_1 \boxtimes T\RR\sp{n})\\
 \, \simeq \, \CI(M_1) \otimes_{\CI(M)} \Big ( r_L \Gamma(M;
 p_1\sp{*}(A_1)) \oplus r_L \Gamma(M; p_2\sp{*}(T\RR\sp{n})) \Big) \\
 \, \simeq \, \CI(M_1) \otimes_{\CI(L)} r_L \Gamma(L; A_1) \, \oplus
 \, \CI(M_1) \otimes_{\CI(M)} r_L \Gamma(M; p_2\sp{*}(T\RR\sp{n})) \,.
\end{multline}
Next, Equation \eqref{eq.product3} gives
\begin{multline*}
 \CI(M_1) \otimes_{\CI(M)} r_L \Gamma(M; p_2\sp{*}(T\RR\sp{n})) \,
 \simeq \, \CI(M_1) \otimes_{\CI(\RR\sp{n})} r_L \Gamma(\RR\sp{n};
 T\RR\sp{n}) \\
 \, \simeq \, \CI(M_1) \otimes_{\CI(R_0)} \CI(R_0)
 \otimes_{\CI(\RR\sp{n})} r_L \Gamma(\RR\sp{n}; T\RR\sp{n}) \\
 \, \simeq \, \CI(M_1) \otimes_{\CI(R_0)} \maV_b(R_0) \, \simeq \,
 \CI(M_1) \otimes_{\CI(R_0)} \Bigl ( \Gamma(R_0; TS\sp{n-1} ) \, \oplus \,
 r\Gamma(R_0; T [0, \infty) ) \Bigr )\,.
\end{multline*}
Let $p_i$, $i = 1, 2, 3$, be the three projections of $L \times
S\sp{n-1} \times [0, \infty)$ onto its components and let $A_1 \to L$,
  $A_2 := T[0, \infty) \to [0, \infty)$, and $A_3 := TS\sp{n-1} \to
      S\sp{n-1}$, be the corresponding three Lie algebroids (with the last
      two being simply the tangent bundles of the corresponding spaces). The above
      calculations then identify $\maW$ with the submodule
\begin{multline}\label{eq.three}
 \maW \, \simeq \, r_L \Gamma(M_1; p_1\sp{*}(A_1)) \, \oplus \, r_L
 \Gamma(M_1; p_2\sp{*}(A_2)) \, \oplus \, \Gamma(M_1; p_3\sp{*}(A_3)) \\
 \subset \, \Gamma(M_1; p_1\sp{*}(A_1)) \, \oplus \, \Gamma(M_1;
 p_2\sp{*}(A_2)) \, \oplus \, \Gamma(M_1; p_3\sp{*}(A_3)) \, \simeq \,
 \Gamma(M_1; A_1 \boxtimes A_2 \boxtimes A_3) \,.
\end{multline}
More precisely, let us denote by $p := (\pi , r_L) : M_1 := [M: L] \to L
\times [0, \infty)$ the natural fibration, where $r_L$ is
  the distance to $L$, as before.  Let $r_L(A_1
  \boxtimes A_2)$ be as in Lemma \ref{lemma.rem.prod}. Then
\begin{equation}\label{eq.four}
 [[A:L]] \, \simeq \, p\pullback \big( r_L(A_1 \boxtimes A_2) \big)
 \,.
\end{equation}
This equation is the local structure result we had anticipated.
It just remains to show that $\maW$ is closed under the Lie
bracket defined on the dense, open subset $M \smallsetminus L
\subset [M:L]$. Indeed, this follows from Equation \eqref{eq.four} and Lemma
\ref{lemma.rem.prod}.
\end{proof}

\begin{definition}\label{def.blow-upA}
 Let us use the notation introduced before Theorem \ref{thm.tame2}
 (in \ref{not.desing}) and in that theorem. 
 Then the Lie algebroid $[[A: L]] = B$ defined in that theorem will be 
 called the {\em desingularization} Lie algebroid of $A$ with
 respect to $L$.
\end{definition}

\begin{remark}
In \cite{Gualtieri}, Gualtieri and Li introduced the \dlp lower
elementary modification\drp\ $[A:B]_{lower}$ of a Lie algebroid $A \to
M$ with respect to a Lie subalgebroid $B \to L$, with $L$ a
submanifold of $M$ and $B \subset A\vert_L$. It is defined by
\begin{equation*}
 \Gamma([A:B]_{lower}) \, := \, \{X \in \Gamma(A), \, X\vert_{L} \in
 \Gamma(B)\,\}\,.
\end{equation*}
One can see right away that their modification is different from ours.
In fact, if $B \neq A\vert_L$, one can see that the right hand side of
the equation is a projective $\CI(M)$-module if, and only if, $L$ is
of codimension one in $M$. In that case (codimension one) one obtains
a vector bundle over the same base $M$, and not over the blow-up
manifold $[M:L]$.
\end{remark}

We have to following consequence of the proof of Theorem
\ref{thm.tame2}.

\begin{corollary}\label{cor.proof}
 Let $\pi : M \to L$ be a vector bundle, $A_1 \to L$ be a Lie 
 algebroid, and let $A = \pi\pullback(A_1)$. Let $A_2 := T[0, \infty)$,
 let $r_L: [M:L] \to [0, \infty)$ be as in \ref{not.desing}, and let
 $p := (\pi, r_L) : [M: L] \to L \times [0, \infty)$ be the natural 
 fibration. Let $r_L(A_1 \boxtimes A_2)$ be as in Lemma \ref{lemma.rem.prod}. Then
\begin{equation*}
 [[A:L]] \, \simeq \, p\pullback \big( r_L(A_1 \boxtimes A_2) \big) \,.
\end{equation*}
\end{corollary}

\begin{proof}
 Locally, this reduces to Equation \eqref{eq.three} (but see also
 Equation \eqref{eq.four}).
\end{proof}

A more general form of Corollary \ref{cor.proof} is the following corollary,
which is a direct consequence of the proof of Theorem \ref{thm.tame2} (see
Equation \eqref{eq.three}).

\begin{corollary}\label{cor.proof2}
 Using the notation of Theorem \ref{thm.tame2} and of its proof, we have
 that $\maW$ be the set of sections $\xi$ of $A$ over $M \smallsetminus L$
 such that, in the neighborhood of every point of $M_1 := [M:L]$, 
 $\xi$ is the restriction 
 of a section of 
 \begin{equation}\label{eq.three.prime}
 r_L \Gamma(M_1; p_1\sp{*}(A_1)) \, \oplus \, r_L
 \Gamma(M_1; p_2\sp{*}(A_2)) \, \oplus \, \Gamma(M_1; p_3\sp{*}(A_3)) \,.
 \end{equation}
\end{corollary}

Here is another of our main results. Recall the definition of the
desingularization $[[\maG: L]]$ of a Lie groupoid $\maG$ along an
$A(\maG)$-tame submanifold $L \subset M$, Definition
\ref{def.blow-upG}.

\begin{theorem}\label{theorem.blow-up}
Let $\maG$ be a Lie groupoid with units $M$ and $L \subset M$ be an
$A(\maG)$-tame submanifold $L \subset M$. Then the Lie algebroid of
$[[\maG: L]]$ is canonically isomorphic to $[[A(\maG): L]]$.
\end{theorem}

\begin{proof} 
Recall the notation introduced in \ref{not.blow}. In particular,
$\maG_2 \subset [[\maG: L]]$ denotes the reduction of $\maG$ to $U_2
:= M \smallsetminus L$.  We have that $\maG_2 = [[\maG:L]]\vert_{U_2}$
as well, and hence,
\begin{equation*}
  A([[\maG: L]])\vert_{U_2} \, = \, A(\maG_2) \, = \, A(\maG)\vert_{U_2} 
  \, = \, [[A(\maG): L]]\vert_{U_2} \,.
\end{equation*}
(This simply means that, up to an isomorphism, nothing changes outside
$L$.)  It suffices then to show that $A([[\maG_1: L]])\vert_{U} =
    [[A(\maG_1): L]]\vert_{U}$, because then
\begin{equation}\label{eq.two}
 A([[\maG: L]])\vert_{U} \, = \, A([[\maG_1: L]])\vert_{U} \, = \,
 [[A(\maG_1): L]]\vert_{U} \, = \, [[A(\maG): L]]\vert_{U} \,.
\end{equation}

Recall that $U$ is the distinguished tubular neighborhood of $L$ used
to defined the desingularization groupoid $[[\maG: L]]$. Also,
$\maG_1$ is the edge modification of $\maG$ and hence $\maG_1$ is the
reduction of $[[\maG:L]]$ to $[U:L]$. See \ref{not.blow} but also
Proposition \ref{prop.glueing}. Let $\pi : U \to L$ denote the
projection. Without loss of generality, we may assume that $M = U$,
that $\pi : M = U \to L$ is a vector bundle, and hence that $\maG =
\pi\pullback (\maH)$. It follows that $A(\maG) \simeq \pi\pullback
A(\maH)$.

We use the notation of Corollary \ref{cor.proof}. Let $p_i$ be the two
projections of $L \times [0, \infty)$ onto its components.  Let $A_2 =
  T[0, \infty)$. Then we have that
\begin{equation}\label{eq.unu}
 A(\maH_{ad}) \, \simeq \, r p_1\sp{*}(A(\maH)) \, \subset \, A(\maH)
 \boxtimes A_2 \,,
\end{equation}
by Equation \eqref{eq.Aad} (see also Equation \eqref{eq.times.t}).
Next, the Lie algebroid of the semi-direct product $\maH_{ad} \rtimes
\RR_+\sp{*}$ is
\begin{equation}\label{eq.doi}
 A(\maH_{ad} \rtimes \RR_+\sp{*}) \, \simeq \, r p_1\sp{*}(A(\maH))
 \oplus r p_2\sp{*}(A_2) \, \simeq \, r(A(\maH) \boxtimes A_2) \,,
\end{equation}
by Equation \eqref{eq.unu} and since the action of $\RR\sp{*}_{+}$ on
$[0, \infty)$ has infinitesimal generator $r \pa_r$, $r \in [0,
    \infty)$.  Finally, the pull-back $p\pullback (\maH_{ad} \rtimes
    \RR_+\sp{*})$ of $\maH_{ad} \rtimes \RR_+\sp{*}$ to $[M : L]$ via
    the projection $p := (\pi, r) : [M : L] \to L \times [0, \infty)$
      is isomorphic to $[[\maG : L]]$.  It has Lie algebroid
      $p\pullback \big ( r(A(\maH) \boxtimes A_2) \big )$. That is,
\begin{multline*}
 A([[\maG:L]]) \, \simeq \, A(p\pullback (\maH_{ad} \rtimes
 \RR_+\sp{*})) \simeq \, p\pullback A(\maH_{ad} \rtimes \RR_+\sp{*})
 \simeq \, p\pullback \big ( r(A(\maH) \boxtimes A_2) \big )\\
 \simeq \, [[A(\maG): L]]\,,
\end{multline*}
where the last isomorphism is by Corollary \ref{cor.proof}, since
$A(\maG) \simeq \pi\pullback A(\maH)$.
\end{proof}

\begin{remark}
The above theorem, Theorem \ref{theorem.blow-up}, is the 
{\em raison d\rp \^etre} for our definition
of a desingularization of a Lie groupoid. Indeed, there are good
reasons in Analysis and Poisson geometry for considering generalized
polar coordinates in the form of coordinates on the blow-up space
$[M:L]$ (think of cylindrical coordinates, which amount to the blow-up
of a line in the three dimensional Euclidean space). This is especially 
convenient when studying the conformal change of metrics that replaces the 
original metric $g$ with $r_L\sp{-2}g$. Some of the
vector fields on the base manifold become singular in the new
coordinates (in our language, they do not {\em lift} to the blow-up).
Multiplying them with the distance function $r_L$ eliminates this
singularity and does not affect too much the resulting differential
operators. At the level of metrics, this corresponds to the  
conformal change of metric $g \to r_L\sp{-2}g$ mentioned above. 
We are thus lead to study vector fields of the form $r_L\maV$, where 
$\maV$ is a given Lie algebra of vector fields (a finitely 
generated, projective module in all our examples). This motivates 
our definition of the desingularization of Lie
algebroids. In Analysis, one may want then to integrate the resulting desingularized
Lie algebroid. Relevant result in this sense were obtained in
\cite{Debord1, NistorJapan}. However, what our results show are that,
if one is given a natural groupoid integrating the {\em original}
(non-desingularized) Lie groupoid (with sections $\maV$), then one can construct starting
from the initial groupoid a new groupoid that will integrate the
desingularized Lie groupoid and at the same time preserve the basic
properties of the original groupoid.
\end{remark}

Related to the above remark, let us mention that it would be interesting
to see if, given a Poisson groupoid structure on $\maG$, whether this
structure lifts to a Poisson groupoid structure on $[[\maG:L]]$
(probably not) or on $[[\maG:L]]_{ni}$ (probably yes, but only under 
some conditions). Some possibly relevant results in this direction can
be found in \cite{camilleArXiv06, Gualtieri, MoerdijkFolBook, Polishchuk}.

\section{Extensions and examples}

This final section contains an extension of the results of the last
subsection to anisotropic desingularizations and an example related 
to the so called \dlp edge calculus.\drp

\subsection{The Lie algebroid of the anisotropic blow-up}
At the level of groupoids, we obtain the following definition, 
thus making Corollary \ref{cor.proof2} the starting point for 
the definition of the anisotropic blow-up.

\begin{definition}\label{def.abup}
 Let us use the notation of Theorem \ref{thm.tame2} and of its proof.
 Let $\maW_{ni}$ be the set of sections $\xi$ of $A$ over $M \smallsetminus L$
 such that, in the neighborhood of every point of $M_1 := [M:L]$, $\xi$ is the restriction 
 of a section of 
 \begin{equation*} 
 \Gamma(M_1; p_1\sp{*}(A_1)) \, \oplus \, r_L
 \Gamma(M_1; p_2\sp{*}(A_2)) \, \oplus \, \Gamma(M_1; p_3\sp{*}(A_3)) \,.
 \end{equation*}
Then $\maW_{in}$ is a projective module over $M_1 := [M:L]$ and a Lie
algebra, and hence it identifies with the sections of a Lie algebroid
$[[A:L]]_{ni} \to [M:L]$.
\end{definition}

Thus the difference between Corollary \ref{cor.proof2} and  
Definition \ref{def.abup} is that we have dropped the factor $r$ on
the first component in Definition \ref{def.abup}.

\begin{theorem} \label{theorem.blow-up2}
Let $\maG$ be a Lie groupoid with units $M$ and $L \subset M$ be an
$A(\maG)$-tame submanifold $L \subset M$. Then the Lie algebroid of
$[[\maG: L]]_{ni}$ is canonically isomorphic to $[[A(\maG): L]]_{ni}$.
In particular, $\Gamma(A([[\maG: L]]))$ is a Lie ideal in
$\Gamma(A([[\maG: L]]_{ni}))$.
\end{theorem}

\begin{proof}
 The proof of the first part is identical to that of Theorem
 \ref{theorem.blow-up}. The second part follows from the fact that
 $r\Gamma(M; p_1\sp{*}(A_1))$ is a Lie ideal in $\Gamma(M;
 p_1\sp{*}(A_1))$.
\end{proof}

At the level of groupoids, one obtaines an action of $[[\maG: L]]_{ni}$
on $[[\maG: L]]$ for which the morphism $\Psi$ is equivariant.

\subsection{The asymptotically hyperbolic modification}
On can consider also the case when $L \subset M$ has a tubular
neighborhood that is not a ball bundle, but something similar.  Let us
assume then that $L \subset M$ is a face of codimension $n$ that is a
manifold with corners in its own. We assume that the neighborhood $U$
of $L$ is such that $U \simeq L \times [0, 1)\sp{n}$, with $\pi : U
  \simeq L \times [0, 1)\sp{n} \to L$ being the projection onto the
first component. Then our methods extend without change in this case,
the result being quite similar. Theorem \ref{thm.tame} and its proof
extend without change to this setting. So do the definition of
$[[\maG:L]]$ and of its anisotropic analog, $[[\maG:L]]_{ni}$, as well
as the results on their structure and Lie algebroid. One just has to
consider $S := L \times (S\sp{n-1} \cap [0, \infty)\sp{n})$.

The case when $L$ is of codimension one is especially relevant, since
it is related to the study of asymptotically hyperbolic spaces, see,
for instance \cite{ammannGrosse, bouclet2006, delay2008, moroianu2010,
  LNGeometric} and the references therein.

\subsection{The local structure of the desingularization for pair groupoids}
\label{ssec.example}
Let us see what these constructions become in the particular, but
important case when we apply these constructions to the pair
groupoid. For the purpose of further reference, let us introduce the
groupoid $\maH_k$ defined as the semidirect product with
$\RR_{+}\sp{*}$ of the adiabatic groupoid $(\RR\sp{k} )\sp{2}_{ad}$ of
the pair groupoid $(\RR\sp{k} )\sp{2}$, that is,
\begin{equation*}
 \maH_{k} \ede (\RR\sp{k} )\sp{2}_{ad} \rtimes \RR_{+}\sp{*} \, = \,
 \RR\sp{k} \times G \, \sqcup \, \bigl(\RR\sp{k} \times (0, \infty)
 \bigr) \sp{2}\,,
\end{equation*}
where $G$ is the semi-direct product $\RR\sp{k}\rtimes \RR_{+}\sp{*}$
and $\sqcup$ denotes again the disjoint union.

\begin{example}\label{ex.local}
 Let us assume that $\maG := \RR\sp{n+k} \times \RR\sp{n+k}$ is the
 pair groupoid and that $L = \RR\sp{k} \times \{0\} \subset
 \RR\sp{n+k} =: M$.  This gives $\maH = L \times L$.  We have $A(\maG)
 = T \RR\sp{n+k}$, and hence $L$ is an $A(\maG)$-tame submanifold.  We
 are, in fact, in the setting of Definition \ref{def.edge}, with $E =
 M$ and $\pi : E \to L$ the natural projection. We have already seen
 that $[M:L] \, \simeq \, S\sp{n-1} \times [0, \infty) \times
   \RR\sp{k}$. By definition $[[\maG:L]] :=
   \pi\pullback(\maH_k)$. Thus
 \begin{multline*}
  [[\maG:L]] \, = \, \bigl( S\sp{n-1} \bigr) \sp{2} \times \, \maH_k\,
  \simeq \, \bigl( S\sp{n-1} \bigr) \sp{2} \times \, \Bigl[ \RR\sp{k}
    \times G \, \sqcup \, \bigl( \RR\sp{k} \times (0, \infty) \bigr)
    \sp{2} \, \Bigr ]\\
  \simeq \, \bigl( S\sp{n-1} \bigr) \sp{2} \times \RR\sp{k} \times G
  \, \sqcup \, \bigl(S\sp{n-1} \times \RR\sp{k} \times (0, \infty)
  \bigr) \sp{2}\,,
 \end{multline*}
 where the first set in the disjoint union corresponds to the
 restriction to $S$, all sets of the form $X\sp{2}$ represent pair
 groupoids, and $G=\RR\sp{k}\rtimes \RR_{+}\sp{*}$, as before.
\end{example}
 
The example of the anisotropic desingularization is very similar.
 
\begin{example} We use the same framework as in the last
example, then
 \begin{equation*}
  [[\maG:L]]_{ni} \, = \, \bigl ( S\sp{n-1} \times \RR\sp{k} \bigr)
  \sp{2} \times \RR_{+}\sp{*} \, \sqcup \, \bigl(S\sp{n-1} \times
  \RR\sp{k} \times (0, \infty) \bigr) \sp{2} \, = \, \bigl ( S\sp{n-1}
  \times \RR\sp{k} \bigr)\sp{2} \times \maT,
 \end{equation*}
 where $\maT := [0, \infty) \rtimes \RR_{+}\sp{*}$, as in
   \ref{sssec.an}.  By writing $G = \RR\sp{k} \times \RR_{+}\sp{*}$ as
{\em sets}, we see that $[[\maG:L]]$ and $[[\maG:L]]_{ni}$ identify as
sets (but not as groupoids!).  In fact, with these identifications,
the natural morphism $\Psi : [[\maG:L]] \to [[\maG:L]]_{ni}$ becomes
$\Psi (s_1, s_1, x_1, x_2, t) = (s_1, s_1, x_1, 0, t)$, if
 \begin{equation*}
  (s_1, s_1, x_1, x_2, t) \, \in \, \bigl( S\sp{n-1} \bigr) \sp{2}
   \times \RR\sp{k} \times G \, = \, \bigl ( S\sp{n-1} \times
   \RR\sp{k} \bigr) \sp{2} \times \RR_{+}\sp{*} \,.
 \end{equation*}
 For $n = 1$, one may want to replace $(S\sp{0})\sp{2}$ with simply
 $(S\sp{0})\sp{2}$, which would make the resulting groupoid
 $d$-connected.
\end{example}

The case of an asymptotically hyperbolic modification is completely
similar.

\begin{example}\label{ex.local.face}
 In case we replace $\RR\sp{n+k}$ in the above example with $\RR\sp{k}
 \times [0, \infty)\sp{n}$, we simply replace the sphere $S\sp{n-1}$
   with $S\sp{n-1} \cap [0, \infty)\sp{k}$.
\end{example}

\begin{example}\label{ex.local.hyp} The simplest case is the one that 
models a true hyperbolic space, that is, $L = \RR\sp{k}$ and $M = L
\times [0, \infty)$. Then we have $[[\maG:L]] = \maH_k$ and
  $[[\maG:L]]_{ni} = \bigl ( \RR\sp{k} \bigr) \sp{2} \times
  \RR_{+}\sp{*} \sqcup \bigl(\RR\sp{k} \times (0, \infty) \bigr)
  \sp{2} = \bigl ( \RR\sp{k} \bigr)\sp{2} \times \maT.$
\end{example}

\subsection{An example: the \lp edge calculus\rp\ groupoid}

Let us conclude with a simple example. That is, we now treat the
desingularization of a groupoid with a smooth set of units over a
smooth manifold. Thus neither the large manifold nor its submanifold
have corners. This example is the one needed to recover the
pseudodifferential calculi of Grushin \cite{grushin71}, Mazzeo
\cite{mazzeo91}, and Schulze \cite{schulzeEdge}.

\begin{remark}\label{rem.smooth}
Let $M$ be a smooth, compact, connected manifold (so $M$ has no
corners).  Recall the path groupoid of $M$, consisting of homotopy
classes of end-point preserving paths $[0, 1] \to M$.  It is a
$d$-simply-connected Lie groupoid integrating $TM$ (that is, its Lie
algebroid is isomorphic to $TM$), so it is the maximal $d$-connected
Lie groupoid with this property. On the other hand, the minimal
groupoid integrating $TM$ is $\maG = M \times M$.  In general, a
$d$-connected groupoid $\maG$ integrating $TM$ will be a quotient of
$\maP(M)$, explicitly described in \cite{Gualtieri} (see also
\cite{MoerdijkMrcun02}).  For analysis questions, it is typically more
natural to choose for $\maG$ the minimal integrating groupoid $M
\times M$. We notice that in analysis one has to use sometimes
groupoids that are not $d$-connected \cite{CarvalhoYu}.
\end{remark}

We shall fix in what follows a smooth, compact, connected manifold $M$
(so $M$ has no corners) and a $d$-connected Lie groupoid $\maG$
integrating the Lie algebroid $TM \to M$.  The following example is
related to some earlier results of Grushin \cite{grushin71}. See also
Coriasco-Schulze \cite{coriascoSchulze},
Guillarmou-Moroianu-Park \cite{moroianu2010}, Lauter-Moroianu
\cite{LauterMoroianu1}, Lauter-Nistor \cite{LNGeometric}, Mazzeo
\cite{mazzeo91}, Schulze \cite{schulzeEdge}, and others, and can be
used to define the so-called \dlp edge calculus\drp.

\begin{example}\label{ex.smooth}
Let $L \subset M$ be an embedded smooth submanifold with tubular
neighborhood $U$ that we identify with the set of vectors of length $<
1$ in $NL$, the normal bundle to $L$ in $M$, as in
\ref{not.submanifold}. We denote by $\pi : S := \pa U \to L$ the
natural projection. Then recall that the blow-up $[M : L]$ of $M$
with respect to $L$ is the disjoint union
\begin{equation*}
 [M : L] \, := \, (M \smallsetminus L) \sqcup S \,,
\end{equation*}
with the topology of a manifold with boundary $S$.  We have that $L$
is automatically $A(\maG) = TM$-tame, so we can define $[[\maG: L]]$
(Definition \ref{def.blow-upG}), which is a Lie groupoid with base
$[M: L]$.
\end{example}

Let us spell out the structure of $[[\maG: L]]$ in order to better
understand the desingularization construction.

\begin{remark}{rem.str.desing}
We continue to use the notation introduced in Example \ref{ex.smooth}.
By the definition of the groupoid $[[\maG: L]]$, the open set $U_0 :=
M \smallsetminus L = [M:L]\smallsetminus S$ is a $[[\maG:
    L]]$-invariant subset and the restriction $[[\maG: L]]_{U_0}$
coincides with the reduction $\maG_{U_0}\sp{U_0}$. In particular, if
$\maG = M \times M$, then $[[\maG: L]]_{U_0} = \maG_{U_0}\sp{U_0} =
U_0 \times U_0$, the pair groupoid. On the other hand, the restriction
of $[[\maG: L]]$ to $S := [M:L] \smallsetminus U_0$ is a fibered
pull-back groupoid defined as follows.  We consider first $TL \to L$,
regarded as a bundle of (commutative) Lie groups. We let
$\RR_{+}\sp{*}$ act on the fibers of $TL \to L$ by dilation and define
the bundle of Lie groups $G_S \to L$ by $ G_{S} := TL \rtimes
\RR_{+}\sp{*} \to L,$ that is, the group bundle over $L$ obtained by
taking the semi-direct product of $TL$, by the action of
$\RR_{+}\sp{*}$ by dilations.  (See also Example \ref{ex.BLG}.) Then $
[[\maG: L]]_S \, := \, \pi\pullback(G)$.  In particular, $[[\maG:
    L]]_S$ does not depend on the choice of integrating groupoid
$\maG$.
\end{remark}

\begin{remark}
Let us choose $\maG := M \times M$.  As mentioned above, if a Lie
group acts on $M$ leaving $L$ invariant, then it will act on $\maG$,
and hence also on $[\maG:L]$, by Proposition
\ref{prop.des.Gamma}. This yields hence also an action of $\Gamma$ on
the edge calculus \cite{grushin71, mazzeo91, schulzeWong2014,
  schulzeEdge}.  See also \cite{krainer2014,
  plamenevskiiBook, seiler1999, vishikGrushin}.
\end{remark}

\begin{remark}
 By choosing $\maG := \maP(M)$, one obtains a \dlp covering edge
 calculus,\drp\ that is, a calculus that is on the universal covering
 manifold $\widetilde M \to M$, is invariant with respect to the group
 of deck transformations, and respects the edge structure along the
 lift of $L$ to $\widetilde M$. See \cite{nistorDocumenta} for
 applications of the covering calculus.
\end{remark}

By iterating this construction as in \cite{schroedinger}, one obtains
integral kernel operators on polyhedral domains.
It would be interesting to extend this example to the
pseudodifferential calculus on manifolds with boundary
\cite{KarstenCR}.

\def\cprime{$'$}


\begin{thebibliography}{10}

\bibitem{schroedinger}
B.~Ammann, Catarina Carvalho, and V.~Nistor.
\newblock Regularity for eigenfunctions of {S}chr\"odinger operators.
\newblock {\em Lett. Math. Phys.}, 101(1):49--84, 2012.

\bibitem{ammannGrosse}
B.~Ammann and N.~Grosse.
\newblock {$L\sp{p}$}-spectrum of the dirac operator on products with
  hyperbolic spaces.
\newblock preprint arXiv:math/1405.2830 [math.DG], 2014.

\bibitem{aln1}
B.~Ammann, R.~Lauter, and V.~Nistor.
\newblock On the geometry of {R}iemannian manifolds with a {L}ie structure at
  infinity.
\newblock {\em Int. J. Math. Math. Sci.}, (1-4):161--193, 2004.

\bibitem{ASkandalis2}
I.~Androulidakis and G.~Skandalis.
\newblock Pseudodifferential calculus on a singular foliation.
\newblock {\em J. Noncommut. Geom.}, 5(1):125--152, 2011.

\bibitem{Arone}
G.~Arone and M.~Kankaanrinta.
\newblock On the functoriality of the blow-up construction.
\newblock {\em Bull. Belg. Math. Soc. Simon Stevin}, 17(5):821--832, 2010.

\bibitem{BMNZ}
C.~Bacuta, A.~Mazzucato, V.~Nistor, and L.~Zikatanov.
\newblock Interface and mixed boundary value problems on {$n$}-dimensional
  polyhedral domains.
\newblock {\em Doc. Math.}, 15:687--745, 2010.

\bibitem{KarstenCR}
K.~Bohlen.
\newblock Boutet de {M}onvel operators on singular manifolds.
\newblock CR Acad Sci Paris (to appear).

\bibitem{bouclet2006}
J.-M. Bouclet.
\newblock Resolvent estimates for the {L}aplacian on asymptotically hyperbolic
  manifolds.
\newblock {\em Ann. Henri Poincar\'e}, 7(3):527--561, 2006.

\bibitem{buneciSurvey}
M.~Buneci.
\newblock Groupoid {$C^\ast$}-algebras.
\newblock {\em Surv. Math. Appl.}, 1:71--98 (electronic), 2006.

\bibitem{CarvalhoYu}
C.~Carvalho and Yu~Qiao.
\newblock Layer potentials {$C^*$}-algebras of domains with conical points.
\newblock {\em Cent. Eur. J. Math.}, 11(1):27--54, 2013.

\bibitem{ConnesBook}
A.~Connes.
\newblock {\em Noncommutative geometry}.
\newblock Academic Press, San Diego, 1994.

\bibitem{coriascoSchulze}
S.~Coriasco and B.-W. Schulze.
\newblock Edge problems on configurations with model cones of different
  dimensions.
\newblock {\em Osaka J. Math.}, 43(1):63--102, 2006.

\bibitem{CuntzQuillen0}
J.~Cuntz and D.~Quillen.
\newblock Excision in bivariant periodic cyclic cohomology.
\newblock {\em Invent. Math.}, 127:67--98, 1997.

\bibitem{damakGeorgescu}
M.~Damak and V.~Georgescu.
\newblock Self-adjoint operators affiliated to {$C^*$}-algebras.
\newblock {\em Rev. Math. Phys.}, 16(2):257--280, 2004.

\bibitem{daugeBook}
Monique Dauge.
\newblock {\em Elliptic boundary value problems on corner domains}, volume 1341
  of {\em Lecture Notes in Mathematics}.
\newblock Springer-Verlag, Berlin, 1988.
\newblock Smoothness and asymptotics of solutions.

\bibitem{Debord1}
C.~Debord.
\newblock Holonomy groupoids of singular foliations.
\newblock {\em J. Differential Geom.}, 58(3):467--500, 2001.

\bibitem{DebordSkandalis}
C.~Debord and G.~Skandalis.
\newblock Adiabatic groupoid, crossed product by {$\mathbb{R}_+^\ast$} and
  pseudodifferential calculus.
\newblock {\em Adv. Math.}, 257:66--91, 2014.

\bibitem{delay2008}
E.~Delay.
\newblock Spectrum of the {L}ichnerowicz {L}aplacian on asymptotically
  hyperbolic surfaces.
\newblock {\em Math. Phys. Anal. Geom.}, 11(3-4):365--379, 2008.

\bibitem{georgescuNistor}
V.~Georgescu and V.~Nistor.
\newblock The essential spectrum of {$N$}-body systems with asymptotically
  homogeneous order-zero interactions.
\newblock {\em C. R. Math. Acad. Sci. Paris}, 352(12):1023--1027, 2014.

\bibitem{grushin71}
V.~V. Gru{\v{s}}in.
\newblock A certain class of elliptic pseudodifferential operators that are
  degenerate on a submanifold.
\newblock {\em Mat. Sb. (N.S.)}, 84 (126):163--195, 1971.

\bibitem{Gualtieri}
M.~Gualtieri and Songhao Li.
\newblock Symplectic groupoids of log symplectic manifolds.
\newblock {\em Int. Math. Res. Not. IMRN}, (11):3022--3074, 2014.

\bibitem{moroianu2010}
C.~Guillarmou, S.~Moroianu, and Jinsung Park.
\newblock Eta invariant and {S}elberg zeta function of odd type over convex
  co-compact hyperbolic manifolds.
\newblock {\em Adv. Math.}, 225(5):2464--2516, 2010.

\bibitem{HigginsMackenzie1}
Ph. Higgins and K.~Mackenzie.
\newblock Algebraic constructions in the category of {L}ie algebroids.
\newblock {\em J. Algebra}, 129(1):194--230, 1990.

\bibitem{HigginsMackenzie2}
Ph. Higgins and K.~Mackenzie.
\newblock Fibrations and quotients of differentiable groupoids.
\newblock {\em J. London Math. Soc. (2)}, 42(1):101--110, 1990.

\bibitem{Joyce}
D.~Joyce.
\newblock On manifolds with corners.
\newblock In {\em Advances in geometric analysis}, volume~21 of {\em Adv. Lect.
  Math. (ALM)}, pages 225--258. Int. Press, Somerville, MA, 2012.

\bibitem{karoubiBook}
Max Karoubi.
\newblock {\em {$K$}-theory}, volume 226 of {\em Grundlehren der Mathematischen
  Wissenschaften}.
\newblock Springer-Verlag, Berlin-New York, 1978.

\bibitem{kottkeMelrose}
C.~Kottke and R.~Melrose.
\newblock Generalized blow-up of corners and fiber products.
\newblock {\em Trans. Amer. Math. Soc.}, 367(1):651--705, 2015.

\bibitem{krainer2014}
T.~Krainer.
\newblock A calculus of abstract edge pseudodifferential operators of type
  $(\rho, \delta)$.
\newblock arXiv:math/1403.6100 [math.AP], 2014.

\bibitem{camilleArXiv06}
C.~Laurent-Gengoux.
\newblock From lie groupoids to resolution of singularities. applications to
  symplectic and poisson resolutions.
\newblock preprint arXiv.org/0610288.pdf [math.DG], 2006.

\bibitem{LauterMoroianu1}
R.~Lauter and S.~Moroianu.
\newblock Fredholm theory for degenerate pseudodifferential operators on
  manifolds with fibered boundaries.
\newblock {\em Comm. Partial Differential Equations}, 26:233--283, 2001.

\bibitem{LNGeometric}
R.~Lauter and V.~Nistor.
\newblock Analysis of geometric operators on open manifolds: a groupoid
  approach.
\newblock In {\em Quantization of singular symplectic quotients}, volume 198 of
  {\em Progr. Math.}, pages 181--229. Birkh\"auser, Basel, 2001.

\bibitem{MackenzieBook1}
K.~Mackenzie.
\newblock {\em Lie groupoids and {L}ie algebroids in differential geometry},
  volume 124 of {\em LMS Lect. Note Series}.
\newblock Cambridge U. Press, Cambridge, 1987.

\bibitem{MackenzieBook2}
K.~Mackenzie.
\newblock {\em General theory of {L}ie groupoids and {L}ie algebroids}, volume
  213 of {\em LMS Lect. Note Series}.
\newblock Cambridge U. Press, Cambridge, 2005.

\bibitem{mantoiuReine}
M.~M{\u{a}}ntoiu.
\newblock {$C^\ast$}-algebras, dynamical systems at infinity and the essential
  spectrum of generalized {S}chr\"odinger operators.
\newblock {\em J. Reine Angew. Math.}, 550:211--229, 2002.

\bibitem{Mantoiu1}
M.~M{\u{a}}ntoiu, Radu Purice, and Serge Richard.
\newblock Spectral and propagation results for magnetic {S}chr\"odinger
  operators; a {$C^*$}-algebraic framework.
\newblock {\em J. Funct. Anal.}, 250(1):42--67, 2007.

\bibitem{marcutzActa}
I.~M{\u{a}}rcu{\c{t}}.
\newblock Rigidity around {P}oisson submanifolds.
\newblock {\em Acta Math.}, 213(1):137--198, 2014.

\bibitem{MargalefHandbook}
J.~Margalef-Roig and E.~Outerelo~Dom{\'{\i}}nguez.
\newblock Topology of manifolds with corners.
\newblock In {\em Handbook of global analysis}, pages 983--1033, 1216. Elsevier
  Sci. B. V., Amsterdam, 2008.

\bibitem{mazzeo91}
R.~Mazzeo.
\newblock {Elliptic theory of differential edge operators. I.}
\newblock {\em Commun. Partial Differ. Equations}, 16(10):1615--1664, 1991.

\bibitem{MoerdijkMrcun02}
I.~Moerdijk and J.~Mr{\v{c}}un.
\newblock On integrability of infinitesimal actions.
\newblock {\em Amer. J. Math.}, 124(3):567--593, 2002.

\bibitem{MoerdijkFolBook}
I.~Moerdijk and J.~Mr{\v{c}}un.
\newblock {\em Introduction to foliations and {L}ie groupoids}, volume~91 of
  {\em Cambridge Studies in Advanced Mathematics}.
\newblock Cambridge University Press, Cambridge, 2003.

\bibitem{Monthubert}
B.~Monthubert.
\newblock Pseudodifferential calculus on manifolds with corners and groupoids.
\newblock {\em Proc. Amer. Math. Soc.}, 127(10):2871--2881, 1999.

\bibitem{NP}
S.~A. Nazarov and B.~A. Plamenevsky.
\newblock {\em Elliptic problems in domains with piecewise smooth boundaries},
  volume~13 of {\em de Gruyter Expositions in Mathematics}.
\newblock Walter de Gruyter \& Co., Berlin, 1994.

\bibitem{nicolaRodinoBook}
F.~Nicola and L.~Rodino.
\newblock {\em Global pseudo-differential calculus on {E}uclidean spaces},
  volume~4 of {\em Pseudo-Differential Operators. Theory and Applications}.
\newblock Birkh\"auser Verlag, Basel, 2010.

\bibitem{frascatti}
V.~Nistor.
\newblock Analysis on singular spaces: Lie manifolds and operator algebras.
\newblock Max Planck preprint 2015, submitted.

\bibitem{nistorDocumenta}
V.~Nistor.
\newblock Higher index theorems and the boundary map in cyclic cohomology.
\newblock {\em Doc. Math.}, 2:263--295 (electronic), 1997.

\bibitem{NistorJapan}
V.~Nistor.
\newblock Groupoids and the integration of {L}ie algebroids.
\newblock {\em J. Math. Soc. Japan}, 52(4):847--868, 2000.

\bibitem{NWX}
V.~Nistor, A.~Weinstein, and Ping Xu.
\newblock Pseudodifferential operators on differential groupoids.
\newblock {\em Pacific J. Math.}, 189(1):117--152, 1999.

\bibitem{plamenevskiiBook}
B.~A. Plamenevski{\u\i}.
\newblock {\em Algebras of pseudodifferential operators}, volume~43 of {\em
  Mathematics and its Applications (Soviet Series)}.
\newblock Kluwer Academic Publishers Group, Dordrecht, 1989.
\newblock Translated from the Russian by R. A. M. Hoksbergen.

\bibitem{Polishchuk}
A.~Polishchuk.
\newblock Algebraic geometry of {P}oisson brackets.
\newblock {\em J. Math. Sci. (New York)}, 84(5):1413--1444, 1997.
\newblock Algebraic geometry, 7.

\bibitem{Pradines}
J.~Pradines.
\newblock Th\'eorie de {L}ie pour les groupo\"\i des diff\'erentiables.
  {C}alcul diff\'erenetiel dans la cat\'egorie des groupo\"\i des
  infinit\'esimaux.
\newblock {\em C.R. Acad. Sci. Paris}, 264:A245--A248, 1967.

\bibitem{Rinehart}
G.~Rinehart.
\newblock Differential forms on general commutative algebras.
\newblock {\em Trans. Amer. Math. Soc.}, 108:195--222, 1963.

\bibitem{schulzeWong2014}
W.~Rungrottheera, B.-W. Schulze, and M.~W. Wong.
\newblock Iterative properties of pseudo-differential operators on edge spaces.
\newblock {\em J. Pseudo-Differ. Oper. Appl.}, 5(4):455--479, 2014.

\bibitem{schroheSymbol}
E.~Schrohe.
\newblock The symbols of an algebra of pseudodifferential operators.
\newblock {\em Pacific J. Math.}, 125(1):211--224, 1986.

\bibitem{schulzeEdge}
B.-W. Schulze.
\newblock Pseudo-differential operators on manifolds with edges.
\newblock In {\em Symposium ``{P}artial {D}ifferential {E}quations''}, volume
  112 of {\em Teubner-Texte Math.}, pages 259--288. Teubner, 1989.

\bibitem{SchulzeBook91}
B.-W. Schulze.
\newblock {\em Pseudo-differential operators on manifolds with singularities},
  volume~24 of {\em Studies in Mathematics and its Applications}.
\newblock North-Holland Publishing Co., 1991.

\bibitem{seiler1999}
J.~Seiler.
\newblock Continuity of edge and corner pseudodifferential operators.
\newblock {\em Math. Nachr.}, 205:163--182, 1999.

\bibitem{vishikGrushin}
M.~I. Vi{\v{s}}ik and V.~V. Gru{\v{s}}in.
\newblock Degenerate elliptic differential and pseudodifferential operators.
\newblock {\em Uspehi Mat. Nauk}, 25(4(154)):29--56, 1970.

\end{thebibliography}
\end{document}